\documentclass[a4paper,reqno, 10pt]{amsart}

\usepackage{amsmath,amssymb,amsfonts,amsthm,mathrsfs}

\usepackage{verbatim,wasysym,cite}

\usepackage[latin1]{inputenc}
\usepackage{microtype}
\usepackage{color,enumitem,graphicx}
\usepackage[colorlinks=true,urlcolor=blue, citecolor=red,linkcolor=blue,
linktocpage,pdfpagelabels, bookmarksnumbered, bookmarksopen]{hyperref}
\usepackage[english]{babel}
\usepackage[symbol]{footmisc}
\renewcommand{\epsilon}{{\varepsilon}}
\usepackage{enumitem}

\numberwithin{equation}{section}
\newtheorem{theorem}{Theorem}[section]
\newtheorem{lemma}[theorem]{Lemma}
\newtheorem{remark}[theorem]{Remark}

\newtheorem{proposition}[theorem]{Proposition}
\newtheorem{corollary}[theorem]{Corollary}

\newcommand{\C}{\mathbb C}

\newcommand{\R}{\mathbb R}
\newcommand{\N}{\mathbb N}

\newcommand{\Z}{\mathbb Z}

\def\({\left(}
\def\){\right)}
\def\<{\left\langle}
\def\>{\right\rangle}

\def\F{\mathcal F}
\def\K{\mathcal K}

\def\L{\mathcal L}
\def\EE{\mathcal E}

\DeclareMathOperator{\RE}{Re}
\DeclareMathOperator{\IM}{Im}

\newcommand{\qtq}[1]{\quad\text{#1}\quad}

\begin{document}

\title[The cubic-quintic NLS with inverse-square potential]{The cubic-quintic nonlinear Schr\"odinger equation with inverse-square potential}

\author[Alex H. Ardila]{Alex H. Ardila}
\address{Universidade Federal de Minas Gerais\\ ICEx-UFMG\\ CEP
  30123-970\\ MG, Brazil} 
\email{ardila@impa.br}

\author[Jason Murphy]{Jason Murphy}
\address{Missouri University of Science and Technology \\ Rolla, MO, USA}
\email{jason.murphy@mst.edu} 

\begin{abstract} We consider the nonlinear Schr\"odinger equation in three space dimensions with a focusing cubic nonlinearity and defocusing quintic nonlinearity and in the presence of an external inverse-square potential.  We establish scattering in the region of the mass-energy plane where the virial functional is guaranteed to be positive.  Our result parallels the scattering result of \cite{KillipOhPoVi2017} in the setting of the standard cubic-quintic NLS.
\end{abstract}


\maketitle

\section{Introduction}
\label{sec:intro}
We consider the long-time behavior of solutions to the cubic-quintic nonlinear Schr\"odinger equation with an inverse-square potential:
\begin{equation}\tag{NLS$_{a}$}\label{NLS}
\begin{cases} 
(i\partial_{t}-\L_{a})u=-|u|^{2}u+|u|^{4}u,\\
u|_{t=0}=u_{0}\in H^{1}(\R^{3}).
\end{cases} 
\end{equation}
Here $u: \R\times\R^{3}\rightarrow \C$, and the operator
\[
\L_{a}=-\Delta+a|x|^{-2}
\]
is defined via the Friedrichs extension with domain $C^{\infty}_{c}(\R^{3}\backslash\left\{0\right\})$. We restrict to $a>-\tfrac14$, which (by the sharp Hardy inequality) guarantees positivity of $\L_a$ and the equivalence
\begin{equation}\label{equiNorms}
\|u\|_{\dot H^1} \sim \|u\|_{\dot H_a^1}:=\|\sqrt{\L_a}u\|_{L^2}.
\end{equation}


The equation \eqref{NLS} has two conserved quantities, namely, the mass and energy:
\[
M(u)=\int_{\R^{3}}|u|^{2}\,dx\qtq{and} E_{a}(u)=\int_{\R^{3}}\tfrac{1}{2}|\nabla u|^{2}+\tfrac{a}{2|x|^{2}}|u|^{2}
-\tfrac{1}{4}|u|^{4}+\tfrac{1}{6}|u|^{6}\,dx.
\] 
By \eqref{equiNorms} and Sobolev embedding, we see that  $M(u_{0})$, $|E_{a}(u_{0})|<\infty$ if $u_{0}\in H^1$.

Our interest in this work is in \emph{scattering} for solutions $u$ to \eqref{NLS}, which means that \begin{equation}\label{Defscat}
\lim_{t\rightarrow\pm\infty}\|u(t)-e^{-it\L_{a}}u_{\pm} \|_{H^{1}}=0\qtq{for some}u_{\pm}\in H^1.
\end{equation}
 A thorough investigation of the scattering problem for the $3d$ cubic-quintic NLS without external potential, i.e.
\begin{equation}\label{NLSfree}
(i\partial_t + \Delta)u = -|u|^2 u + |u|^4 u,
\end{equation} 
was previously carried out in \cite{KillipOhPoVi2017}.  In particular, scattering was established in the region of the mass-energy plane in which the virial functional (cf. \eqref{E:VFnl} below) is guaranteed to be positive. This region was further extended in \cite{KillipMurphyVisan2020}, still relying on the virial identity in a fundamental way.  Our goal in this work is to initiate the study of the effect of an external potential on the dynamics of solutions for the cubic-quintic model.  Our main result is analogous to that of \cite{KillipOhPoVi2017}, establishing scattering in the region in the mass-energy plane where the virial functional is positive. 

While \eqref{NLS} is globally well-posed in $H^1$ (see Theorem~\ref{Th1}), we do not expect scattering to hold for arbitrary $H^1$ data.  Indeed, in the case of an attractive potential ($a<0$), we can construct a family of solitary wave solutions as optimizers for certain Gagliardo--Nirenberg--H\"older inequalities (see \eqref{InvGN} below).  Our main result instead proves scattering in the region of the mass-energy plane in which the \emph{virial functional} 
\begin{equation}\label{E:VFnl}
V_{a}(f):=\| f\|^{2}_{\dot{H}_{a}^{1}}+\|f\|^{6}_{L_{x}^{6}}
-\tfrac{3}{4}\|f\|^{4}_{L_{x}^{4}}
\end{equation}
is guaranteed to be positive. To make this precise, we first introduce the quantity 
\begin{equation*}
\EE_{a}(m):=\inf\{E_{a\wedge 0}(f): f\in H^{1}(\R^{3}),\, M(f)=m \,\, \mbox{and}\,\, V_{a\wedge 0}(f)=0 \}, 
\end{equation*}
where $a\wedge b:=\min\left\{a,b\right\}$.  We then define the region $\K_{a}\subset \R^{2}$ by 
\begin{equation}
\label{Virial}
\K_{a}:= \{(m,e): 0<m<M(Q_{1, a\wedge 0}) \,\, \mbox{and} \,\, 0<e<\EE_{a}(m)\},
\end{equation}
where $Q_{1, a\wedge 0}$ is an optimizer of \eqref{InvGN} with $\alpha=1$ (see Section~\ref{S:SGNH}).  By Corollary \ref{GNSharp} below, we may also write $M(Q_{1, a\wedge 0})=\(\frac{8}{3}\)^{2}C^{-2}_{1, a\wedge 0}$, where $C_{\alpha,a}$ denotes the sharp constant in \eqref{InvGN}. 

Our main result is the following theorem.  We note that the lower bound on $a$ arises in the local theory for the equation (see e.g. \cite{KillipMiaVisanZhangZheng}). 

\begin{theorem}\label{MainTheorem}
Assume $a>-\frac{1}{4}+\frac{1}{25}$. Let $u_{0}\in H^{1}(\R^{3})$ satisfy $(M(u_{0}),E_{a}(u_{0}))\in \K_{a}$.  Then the corresponding solution $u$ of \eqref{NLS} with initial data $u_{0}$ is global and satisfies
\[
\|  u \|_{L^{10}_{t, x}(\R\times\R^{3})}\leq C(M(u_{0}), E_{a}(u_{0})).
\]
In particular, the solution scatters in $ H^{1}(\R^{3})$ in the sense of \eqref{Defscat}.
\end{theorem}

Theorem~\ref{MainTheorem} parallels the scattering result obtained in \cite{KillipOhPoVi2017} for the standard cubic-quintic equation.  As in \cite{KillipOhPoVi2017}, we can give a more precise description of the scattering region.  To do so, we first introduce the following sharp $\alpha$-Gagliardo--Nirenberg--H\"older inequality:
\begin{equation}\label{InvGN}
\|f\|^{4}_{L_{x}^{4}}\leq 
C_{\alpha, a}\|f\|_{L_{x}^{2}}\|f\|^{\frac{3}{1+\alpha}}_{\dot{H}^{1}_{a}}\|f\|^{\frac{3\alpha}{1+\alpha}}_{L_{x}^{6}}, \quad \alpha\in(0,\infty).
\end{equation}
The optimization problem for \eqref{InvGN} leads to the stationary problem
\[
\L_{a}Q_{\alpha,a}+Q_{\alpha,a}^{5}-Q^{3}_{\alpha,a}+\omega Q_{\alpha,a}=0.
\]
If this problem admits a solution $Q_{\alpha,a}$, then we obtain a (nonscattering) solution to \eqref{NLS} given by $u(t)=e^{i\omega t}Q_{\alpha,a}$.  We will prove that optimizers for \eqref{InvGN} exist when $a\leq0$, while for $a>0$ we obtain $C_{\alpha,a}=C_{\alpha,0}$ but equality is never attained.  Denoting by $Q_{1, a\wedge 0}$ any optimizer of \eqref{InvGN} with $\alpha=1$, we then have the following: 
\begin{itemize}
\item[(i)] For $a>-\frac{1}{4}$ we have the inclusion $\K_{a}\subseteq \K_{0}$ (see  Corollary~\ref{Compa22}). Moreover, by definition,
$\K_{a}=\K_{0}$ for $a\geq 0$. 
\item[(ii)] We set 
\[
S_{a}(x):=\tfrac{1}{\sqrt{2}}Q_{1, a\wedge 0}(\tfrac{\sqrt{3}}{2}x).
\]
Direct calculation shows that 
\[
M(S_{a})<M(Q_{1, a\wedge 0}),\quad V_{a\wedge 0}(S_{a})=0,\qtq{and}E_{a\wedge 0}(S_{a})>0.
\]
Then $\EE_{a}(\cdot)$ satisfies:
\[\EE_{a}(m)
\begin{cases} 
=\infty & {m\in (0, M(S_{a}))},\\
\in (0, E_{a\wedge 0}(S_{a})]& {m\in [M(S_{a}), M(Q_{1, a\wedge 0}))},\\
= 0  & {m=M(Q_{1, a\wedge 0})}
\end{cases} 
\]
and $\EE_{a}$ is strictly decreasing and lower semicontinuous on the interval $[M(S_{a}), M(Q_{1, a\wedge 0})]$ (see Theorem~\ref{Vscatte2}). 
\end{itemize} 

We may depict the region $\K_a$ in the following figure:
\vspace{0.1mm}

\begin{figure}[h]\label{Figu1}
\caption{Mass/Energy Plane}
\centering
\includegraphics[width=0.8 \textwidth]{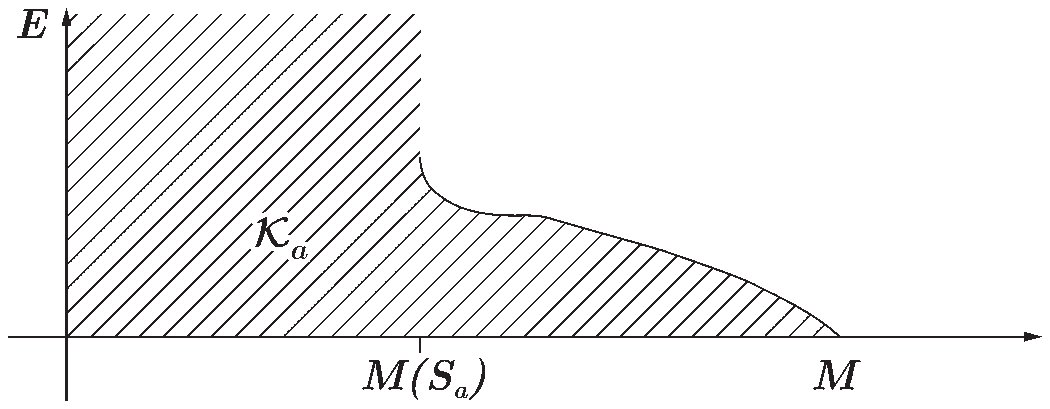}
\end{figure}

%

As mentioned above, Theorem~\ref{MainTheorem} is an analogue of the main scattering result in \cite{KillipOhPoVi2017}.  Accordingly, much of what we do parallels the overall argument in \cite{KillipOhPoVi2017}, utilizing various results from \cite{KMVZZ2018, KillipMiaVisanZhangZheng, KillipMurphyVisanZheng2017} as well.  In order to minimize replication of existing results, we have opted to omit certain proofs throughout the paper, directing the reader instead to the appropriate results in the works just cited.  In particular, this allows us to focus attention on the parts of the argument where new ideas are needed.

Our choice of the inverse-square potential was motivated by several factors.  First, the tools needed for the analysis (e.g. Strichartz estimates, well-posedness and stability theory, Littlewood--Paley theory adapted to $\L_a$, concentration-compactness tools, and more) have already been established for this model (see e.g. \cite{KMVZZ2018, KillipMiaVisanZhangZheng, BPSTZ}). Indeed, these tools have been applied in several instances in the case of a single power nonlinearity (see e.g. \cite{KillipMiaVisanZhangZheng, LU20183174, YANG2020124006, KillipMurphyVisanZheng2017, MMZ2020, MMZWave}).  In fact, a large-data scattering theory for the underlying quintic equation (obtained in \cite{KillipMiaVisanZhangZheng}) is a necessary ingredient for the present work, as we explain below.  Much of the success in treating large-data problems for the NLS with an inverse-square potential is due to the fact that in this case, the potential shares the same scaling as the Laplacian.  This ultimately manifests in the ability to derive virial identities and estimates that parallel the case of NLS without potential.  A final appealing feature of working with an inverse-square potential is that we may consider the effect of both attractive and repulsive potentials simply by varying the sign of the coupling constant.

The proof of Theorem~\ref{MainTheorem} follows the concentration-compactness approach, based on an induction scheme in the mass-energy plane analogous to that in \cite{KillipOhPoVi2017}.  The key components of the proof are therefore (i) the variational analysis needed to define and describe the region in the mass-energy plane corresponding to positive virial and (ii) the construction of a minimal blowup solution (under the assumption that Theorem~\ref{MainTheorem} fails). With such a solution in hand, we reach a contradiction by an application of the localized virial argument.

An interesting aspect of the analysis arises in step (ii) above.  In this step, one is interested in obtaining compactness for a sequence of initial data corresponding to solutions with diverging space-time norms.  The key to precluding dichotomy is to argue by contradiction and then develop a `nonlinear profile decomposition' for the sequence of solutions.  This, in turn, requires the construction of scattering solutions to \eqref{NLS} corresponding to each profile appearing in the linear profile decomposition for the data. However, these profiles may be parametrized by nontrivial scaling and translation parameters, while both the scaling and translation symmetries are broken in the model \eqref{NLS}.  Thus, as has already been observed in works such as \cite{KillipOhPoVi2017, KillipMiaVisanZhangZheng}, the key to constructing the nonlinear profiles is to appeal to the stability theory, using solutions to suitable `limiting' model equations as approximate solutions to the full equation.  In the setting of \eqref{NLS}, we found that we must contend with three distinct scenarios:
\begin{itemize}
\item For unit-scale profiles centered far from the origin, we approximate with solutions to the standard cubic-quintic NLS without potential \eqref{NLSfree}, relying on the scattering result of \cite{KillipOhPoVi2017} and the fact that $\mathcal{K}_a\subset \mathcal{K}_0$.
\item For small-scale profiles centered near the origin, we approximate with solutions to the quintic NLS with inverse-square potential \eqref{CriticalNLS}, relying on the scattering result of \cite{KillipMiaVisanZhangZheng}.
\item For small-scale profiles centered far from the origin (relative to their length scale), we approximate with solutions to the standard quintic NLS without potential \eqref{3dquintic}, relying on the scattering result of \cite{TaoKell2008}.
\end{itemize} 
Roughly speaking, we see that at small scales the cubic nonlinearity may be neglected, while far from the origin the potential may be neglected.  For more details, see Proposition~\ref{P:embedding}.

In \cite{KillipMurphyVisan2020}, the authors additionally succeeded in proving scattering for the model \eqref{NLSfree} in an open neighborhood of $\K_0$.  In particular, this neighborhood contains any part of the boundary that is not represented by a solitary wave; it also yields a strictly larger mass threshold for scattering (without any constraint on the energy). It is natural to consider the analogous problem in the present setting, at least in the case $a<0$ (when solitary waves are present).  Presently, however, certain ingredients are missing (e.g. uniqueness of ground states for the underlying stationary problem) that ultimately leave an analogous result mostly out of reach.  Thus, we have opted to leave the investigation of scattering beyond the region $\K_a$ for a future work.  Similarly, the behavior near the boundary of $\K_a$ in the regime $a>0$ is an interesting direction for future investigation.

The rest of this paper is organized as follows:
\begin{itemize}
\item In Section~\ref{S1:preli}, we set up some notation and collect some preliminary lemmas.  This includes some well-posedness and stability results for \eqref{NLS}. 
\item In Section~\ref{sec:VarationalGNI}, we study the problem of the existence of optimizers for the $\alpha$-Gagliardo--Nirenberg--H\"older inequality.  In particular, we study the variational problem for $\mathcal{E}_a(m)$ introduced above and prove the properties of $\K_a$ described above. 
\item In Section~\ref{Sec:CMBS}, we carry out the construction of minimal blowup solutions under the assumption that Theorem~\ref{MainTheorem} fails.   
\item In Section~\ref{Su8} we use the localized virial argument to preclude the possibility of minimal blowup solutions, thus completing the proof of Theorem~\ref{MainTheorem}.
\end{itemize}

\subsection*{Acknowledgements} J. M. was supported by a Simons Collaboration Grant.

\section{Preliminaries}\label{S1:preli} 
We write $A\lesssim B$ or $A=O(B)$ when $A\leq CB$ for some $C>0$. If  $A \lesssim B \lesssim A$ we write $A\sim B$. We write $a\wedge b=\min\left\{a,b\right\}$. For a function 
$u:I\times \R^{3}\rightarrow \C$ we use the notation
\[ 
\|  u \|_{L_{t}^{q}L^{r}_{x}(I\times \R^{3})}=\|  \|u(t) \|_{L^{r}_{x}(\R^{3})}  \|_{L^{q}_{t}(I)}
 \]
with $1\leq q\leq r\leq\infty$. When $q=r$ we abbreviate $L_{t}^{q}L^{r}_{x}$ by $L_{t,x}^{q}$.

We define the Sobolev spaces associated with $\L_{a}$ via
\[
\|f\|_{\dot{H}^{s, r}_{a}(\R^{3})}=\|(\L_{a})^{\frac{s}{2}}\|_{L^{r}_{x}(\R^{3})}
\quad \text{and}\quad 
\|f\|_{{H}^{s,r}_{a}(\R^{3})}=\|(1+\L_{a})^{\frac{s}{2}}\|_{L^{r}_{x}(\R^{3})}.
\]
We abbreviate $\dot{H}^{s}_{a}(\R^{3})=\dot{H}^{s,2}_{a}(\R^{3})$ and ${H}^{s}_{a}(\R^{3})={H}^{s,2}_{a}(\R^{3})$.
Given $p\in [1, \infty]$, we let $p'\in [1, \infty]$ denote the H\"older dual of $p$.

To state the results that follow, it is also convenient to define
\[
\rho:=\tfrac{d-2}{2}-\big[\(\tfrac{d-2}{2}\)^{2}+a\big]^{\frac{1}{2}}
\quad \text{and}\quad 
q_{0}:=\begin{cases} 
\infty \quad \text{if $a\geq 0$},\\
\frac{d}{\rho}\quad \text{if $-\(\frac{d-2}{2}\)^{2}\leq a<0$}.
\end{cases} 
\]

The following lemma from \cite{KMVZZ2018} summarizes the situation regarding equivalence of Sobolev spaces: 
\begin{lemma}\label{EquiSobolev}
Fix $d\geq 3$, $a\geq \( \frac{d-2}{2}\)^{2}$ and $0<s<2$. If $1<p<\infty$ satisfies 
$\frac{\sigma+s}{d}<\frac{1}{p}<\min\left\{1, \frac{d-\sigma}{d}\right\}$, then 
\[
\|(-\Delta)^{\frac{s}{2}}f\|_{L_{x}^{p}}\lesssim_{d,p,s}\|\L_{a}^{\frac{s}{2}}f\|_{L_{x}^{p}}
\quad \text{for $f\in C^{\infty}_{c}(\R^{d}\setminus\left\{{0}\right\})$}.
\]
If $\max\left\{\frac{s}{d}, \frac{\sigma}{d}\right\}<\frac{1}{p}<\min\left\{1, \frac{d-\sigma}{d}\right\}$, then 
\[
\|\L_{a}^{\frac{s}{2}}f\|_{L_{x}^{p}}\lesssim_{d,p,s}\|(-\Delta)^{\frac{s}{2}}f\|_{L_{x}^{p}}
\quad \text{for  $f\in C^{\infty}_{c}(\R^{d}\setminus\left\{{0}\right\})$}.
\]
In particular, if $a>-\frac{1}{4}+\frac{1}{25}$, then 
\[
\|\L_{a}^{\frac{s}{2}}f\|_{L_{x}^{p}}\sim\|(-\Delta)^{\frac{s}{2}}f\|_{L_{x}^{p}}
\quad \text{for all  $\tfrac{6}{5}\leq p\leq \tfrac{30}{13}$}.
\]
\end{lemma}


We will need some Littlewood--Paley theory adapted to $\L_a$ (as developed in \cite{KMVZZ2018}).  Let $\phi\in C^{\infty}_{c}(\R^{3})$ be a smooth positive radial function obeying 
$\phi(x)=1$ if $|x|\leq 1$  and $\phi(x)=0$ if $|x|\geq \frac{11}{10}$. For $N\in 2^{\Z}$, we define 
\[
\phi_{N}(x):=\phi(x/N)  \quad \text{and}\quad      \psi_{N}(x)=\phi_{N}(x)-\phi_{N/2}(x).
\]
We define the Littlewood-Paley projections
\begin{align*}
f_{\leq N}&:=P^{a}_{\leq N}f:=\phi_{N}(\sqrt{\L_{a}}), \quad f_ {N}:=P^{a}_{N}f:=\psi_{N}(\sqrt{\L_{a}}),\\
&\quad \text{and}\quad f_{> N}:=P^{a}_{> N}f:=(I-P^{a}_{\leq N})f
\end{align*}

The Littlewood-Paley projections obey the following estimates.
\begin{lemma}[Bernstein inequalities, \cite{KMVZZ2018}]
Let $s\in \R$. For $q^{\prime}_{0}<q\leq r<q$ and $f: \R^{d}\to \C$ we have
\begin{align*}
	\| P^{a}_{N}f   \|_{L_{x}^{r}}&\lesssim N^{\frac{d}{q}-\frac{d}{r}}\| P^{a}_{N}f   \|_{L_{x}^{q}},\\
	\| P^{a}_{\geq N}f   \|_{L_{x}^{r}}&\lesssim N^{-s}\| \L_{a}^{\frac{s}{2}} P^{a}_{\geq N}f   \|_{L_{x}^{r}},\\
	N^{s}\| P^{a}_{N}f   \|_{L_{x}^{r}}&\sim \|\L^{\frac{s}{2}}_{a} P^{a}_{N}f   \|_{L_{x}^{r}}.
\end{align*}
\end{lemma}

\begin{lemma}[Square function estimate, \cite{KMVZZ2018}] Let $0\leq s<2$ and $q^{\prime}_{0}<r<q_{0}$. Then we have
\[
\Big\|\Big(  \sum_{N\in 2^{\Z}}N^{2s}| P^{a}_{N}f |^{2}\Big)^{\frac{1}{2}}\Big\|_{L^{r}_{x}}
\sim
\|\L^{\frac{s}{2}}_{a} f   \|_{L_{x}^{r}}.
\]
\end{lemma}

We also import the following local smoothing result for the propagator $e^{-it\L_{a}}$; see \cite[Corollary 2.9]{KillipMiaVisanZhangZheng}.  This result is used precisely once in the paper, namely, to control an error term in an approximate solution in the construction of minimal blowup solutions (see \eqref{Bound33}).

\begin{lemma}\label{LocalSmoo}
Let $a>-\frac{1}{4}+\frac{1}{25}$. Given $\phi\in \dot{H}^{1}_{a}(\R^{3})$,
\begin{align*}
	 \|\nabla  e^{-it\L_{a}}\phi   \|_{L^{5}_{t}L_{x}^{\frac{15}{8}}([\tau-T, \tau+T]\times\left\{|x-z|\leq R\right\})}
	&\lesssim T^{\frac{29}{320}}R^{\frac{51}{160}}\| e^{-it\L_{a}}\phi  \|^{\frac{1}{32}}_{L^{10}_{t, x}(\R\times\R^{3})}
	\|\phi\|^{\frac{31}{32}}_{\dot{H}^{1}_{x}}\\
	&+T^{\frac{29}{280}}R^{\frac{41}{140}}\| e^{-it\L_{a}}\phi  \|^{\frac{1}{28}}_{L^{10}_{t, x}(\R\times\R^{3})}
	\|\phi\|^{\frac{27}{28}}_{\dot{H}^{1}_{x}}
\end{align*}
uniformly in $\phi$ and the parameters $R$, $T>0$, $z\in \R^{3}$ and $\tau\in\R$.

\end{lemma}

Finally, we have the following global-in-time Strichartz estimates.
\begin{lemma}[Strichartz estimates, \cite{BPSTZ}]\label{STRICH}
Fix $a>-\frac{1}{4}$. Then the solution $u$ of $(i\partial_{t}-\L_{a}) u=F$ on an interval $I\ni t_{0}$
obeys
\[\| u  \|_{L_{t}^{q}L^{r}_{x}(I\times \R^{3})} \lesssim \| u(t_{0})\|_{L^{2}_{x}(\R^{3})}+ \| F\|_{L_{t}^{\tilde{q}'}L^{\tilde{r}'}_{x}(I\times \R^{3})}, \]
where $2\leq \tilde{q},\tilde{r}\leq\infty$ with 
$\frac{2}{q}+\frac{3}{r}=\frac{2}{\tilde{q}}+\frac{3}{\tilde{r}}=\frac{3}{2}$ and $(q, \tilde{q})\neq (2,2)$.
\end{lemma}

Throughout the paper we use the notation:
\[
S^{s}_{a}(I)=L^{2}_{t}H^{s,6}_{a}\cap L^{\infty}_{t}H^{s}_{a}(I\times\R^{3})
\quad \text{and}\quad 
\dot{S}^{s}_{a}(I)=L^{2}_{t}\dot{H}^{s,6}_{a}\cap L^{\infty}_{t}\dot{H}^{s}_{a}(I\times\R^{3}).
\]

\subsection{Global well-posedness and stability}\label{Sec:smalldata}

In this section we present the well-posedness theory for \eqref{NLS} in the space $H_a^1$.  First, we have the following global well-posedness result:

\begin{theorem}[Global well-posedness]\label{Th1}
Given $a>-\frac{1}{4}+\frac{1}{25}$ and $u_{0}\in H_{a}^{1}(\R^{3})$, the corresponding solution $u\in C_t H_x^1$ of \eqref{NLS} exists globally in time.  Moreover, we have the conservation of energy and mass, i.e. 
\[
E_{a}(u(t))\equiv E_{a}(u_{0})\qtq{and}  M(u(t))\equiv M(u_{0})\quad \text{for all $t\in \R$.}
\]
\end{theorem}

The corresponding result for the standard cubic-quintic NLS may be found in \cite{Zhang}.  The ingredients needed there are:
\begin{itemize}
\item[(i)] global well-posedness for the defocusing $3d$ quintic NLS
\begin{equation}\label{3dquintic}
(i\partial_t + \Delta)u = |u|^4 u
\end{equation}
(cf. \cite{KiiVisan2008, Bourgain1999, TaoKell2008}), 
\item[(ii)] a stability-type result (referred to as `good local well-posedness' in \cite{Zhang}), and 
\item[(iii)] \emph{a priori} $\dot H^1$ bounds.  
\end{itemize} 
Theorem~\ref{Th1} follows from the fact that we have all of these ingredients in the present setting as well.  In particular, the analogue of (i) was established in \cite{KillipMiaVisanZhangZheng}.  We state the result as follows:
\begin{theorem}[Scattering for the quintic NLS with inverse-square potential]\label{CriticalWP}
\text{ } \\Given $a>-\frac{1}{4}+\frac{1}{25}$ and $u_{0}\in \dot{H}_{a}^{1}(\R^{3})$ there exists a unique global solution $u\in C(\R, \dot{H}_{x}^{1}(\R^{3}))$ to
\begin{equation}\label{CriticalNLS}
\begin{cases} 
(i\partial_{t}-\L_{a})u=|u|^{4}u,\quad (t,x)\in \R\times\R^{3},\\
u(0)=u_{0}\in \dot{H}_{x}^{1}(\R^{3}).
\end{cases} 
\end{equation}
Furthermore, we have the following space-time bound
\[
\int_{\R}\int_{\R^{3}}|u(t,x)|^{10}\,dx\,dt\leq C({\|u_{0}\|_{\dot{H}_{x}^{1}}}).
\]
\end{theorem}

\begin{remark}\label{persistenceCritical}
By Theorem \ref{CriticalWP} and persistence of regularity, one can show that  the global solution $u$ in Theorem \ref{CriticalWP} satisfies
\[
\|\L_{a}^{\frac{1}{2}} u\|_{L_{t}^{q}L^{r}_{x}(\R\times\R^{3})}
\leq C(\|\L_{a}^{\frac{1}{2}} u_{0}\|_{{L}^{2}}).
\]
For all admissible pairs $\frac{2}{q}+\frac{3}{r}=\frac{3}{2}$ with $2<q\leq \infty$. Moreover, if $u(0)=u_{0}\in {H}_{x}^{1}(\R^{3})$,
the we also have
\[
\|u\|_{L_{t}^{q}L^{r}_{x}(\R\times\R^{3})}
\leq C(\|u_{0}\|_{{L}^{2}}).
\]
\end{remark}

 Given Theorem~\ref{CriticalWP} and the Strichartz estimates adapted to the inverse-square potential (cf. Lemma~\ref{STRICH}), the arguments of \cite{Zhang} apply equally well to establish the analogue of (ii) in the setting of \eqref{NLS}.  Finally, the kinetic energy control follows as in \cite{Zhang} as well.  In particular, one observes that by Young's inequality,
\[
\tfrac{1}{4}|u|^4 \leq \tfrac{3}{8}|u|^2 + \tfrac{1}{6}|u|^6, 
\qtq{so that}
\tfrac12 \|u(t)\|_{\dot H^1}^2 \leq E_a(u) + \tfrac{3}{8}M(u).
\]
uniformly over the lifespan of $u$, yielding (iii).


In addition to global well-posedness in $H^1$, we will need a following results establishing scattering in $H^1$ for sufficiently small initial data, along with a persistence of regularity result and a stability result.  All of these are analogues of results in \cite[Section~6]{KillipOhPoVi2017}.  As the proofs rely primarily on Strichartz estimates, which are readily available in the setting of the inverse-square potential, we omit them here. 

\begin{proposition}[Small data scattering]\label{SDC}
Let $a>-\frac{1}{4}+\frac{1}{25}$ and $u_{0}\in H_{a}^{1}(\R^{3})$. There exists $\delta>0$ such that if $\|u_0\|_{H_a^1}<\delta$, then the corresponding solution $u$ of \eqref{NLS} is global and scatters, with
\[
\|u\|_{L^{10}_{t,x}(\R\times\R^{3})} \lesssim \|\sqrt{\L_{a}}u_{0}\|_{L^{2}(\R^{3})}.
\]
\end{proposition}

\begin{remark}[Persistence of regularity]\label{PRegularity}
Suppose that $u: \R\times \R^{3}\rightarrow \C$ is a solution to \eqref{NLS} such that 
$S:=\|u\|_{L^{10}_{t,x}(\R\times\R^{3})}<\infty$. 
Then for $t_{0}\in \R$ we have
\begin{align*}
\|u\|_{S_{a}^{0}(\R)}&\leq C(S, M(u_{0}))\|u(t_{0})\|_{L^{2}_{x}(\R^{3})}, \\
\|u\|_{L^{10}_{t}\dot{H}_{a}^{1,\frac{30}{13}}(\R\times\R^3)}
&\leq C(S, M(u_{0}))\|u(t_{0})\|_{\dot{H}^{1}_{a}(\R^{3})}.
\end{align*}

\end{remark}

\begin{lemma}[Stability]\label{StabilityNLS}
Fix $a>-\frac{1}{4}+\frac{1}{25}$. Let $I\subset \R$ be a time interval containing $t_{0}$ and  let $\tilde{u}$ satisfy
\[ 
(i\partial_{t}-\L_{a}) \tilde{u}=-|\tilde{u}|^{2}\tilde{u}+|\tilde{u}|^{4}\tilde{u}+e, \quad 
\tilde{u}(t_{0})=\tilde{u}_{0}
\]
on $I\times\R^3$ for some $e:I\times\R^{3}\rightarrow \C$. Assume the conditions
\[
\| \tilde{u}  \|_{L_{t}^{\infty}H^{1}_{a}(I\times\R^{3})}\leq E\qtq{and} 
	\| \tilde{u}  \|_{L_{t,x}^{10}(I\times\R^{3})}\leq L
\]
for some $E$, $L>0$. Let $t_{0}\in I$ and $u_{0}\in H_{a}^{1}(\R^{3})$ such that $\|u_{0}\|_{L_{x}^{2}}\leq M$
for some  positive constant $M$.  
Assume also the smallness conditions 
\[
	\|u_{0}-\tilde{u}_{0} \|_{\dot{H}^{1}_{a}}\leq \epsilon\qtq{and}	\|\sqrt{\L_{a}} e  \|_{N(I)}\leq \epsilon,
\]
for some $0<\epsilon<\epsilon_{0}=\epsilon_{0}(\mbox{A,L,M})>0$, where
\[
N(I)=L^{1}_{t}L^{2}_{x}(I\times \R^{3})+L^{2}_{t}L^{\frac{6}{5}}_{x}(I\times \R^{3})+L^{\frac{5}{3}}_{t}L^{\frac{30}{23}}_{x}(I\times \R^{3}).
\]
Then there exists a unique global solution $u$ to Cauchy problem \eqref{NLS} with initial data $u_{0}$ at the time $t=t_{0}$ satisfying
\[
	\|u-\tilde{u}\|_{\dot{S}_{a}^{1}(I)}\leq C(E,L,M)\epsilon\qtq{and}	\|u\|_{\dot{S}_{a}^{1}(I)}\leq C(E,L,M).
\]
Moreover,
\[
\text{if}\quad\|u_{0}-\tilde{u}_{0} \|_{\dot{H}^{\frac{3}{5}}_{x}}+\||\nabla|^{\frac{3}{5}} e  \|_{N(I)}<\epsilon,\qtq{then}
\|u-\tilde{u}\|_{\dot{S}_{a}^{\frac{3}{5}}(I)}\leq C(E,L,M)\epsilon.
\]
\end{lemma}

\section{Variational analysis}
\label{sec:VarationalGNI}

\subsection{Sharp Gagliardo--Nirenberg--H\"older inequality}\label{S:SGNH} In this section, we consider the following $\alpha$-Gagliardo-Nirenberg-H\"older inequality:
\begin{equation}\label{InvGNew}
\|f\|^{4}_{L^{4}}\leq 
C_{\alpha,a}\|f\|_{L^{2}}\|f\|^{\frac{3}{1+\alpha}}_{\dot{H}^{1}_{a}}\|f\|^{\frac{3\alpha}{1+\alpha}}_{L^{6}}.
\end{equation}

We prove the following:

\begin{theorem}\label{TheInverGN}
Let $\alpha\in(0,\infty)$ and $a>-\frac{1}{4}$. Define
\begin{equation}\label{GNI}
C^{-1}_{\alpha,a}:=\inf_{ f\in {H}^{1}_{a}\setminus\left\{0\right\}}
\frac{\|f\|_{L^{2}}\| f\|^{\frac{3}{1+\alpha}}_{\dot{H}^{1}_{a}}\|f\|^{\frac{3\alpha}{1+\alpha}}_{L^{6}}}
{\|f\|^{4}_{L^{4}}}.
\end{equation}
Then $C_{\alpha,a}\in (0,\infty)$ and the following statements hold.
\begin{itemize}
\item[(i)] Assume $a\leq0$. Then the infimum \eqref{GNI} is attained by a function 
$Q_{\alpha,a}\in {H}^{1}_{a}\setminus\left\{0\right\}$, which is a non-negative, radial solution of the 
stationary problem
\begin{equation}\label{EllipNLS}
\L_{a}Q_{\alpha,a}+Q_{\alpha,a}^{5}-Q^{3}_{\alpha,a}+\omega Q_{\alpha,a}=0
\end{equation}
for some $\omega\in (0,\tfrac{3\alpha}{16(1+\alpha)})$. Furthermore, $Q_{\alpha,a}$ satisfies the identity 
\[
{\|Q_{\alpha, a}\|^{6}_{L^{6}}}={\alpha}{\|Q_{\alpha, a}\|^{2}_{\dot{H}^{1}_{a}}}.
\]
\item[(ii)] Assume $a>0$. Then $C_{\alpha,a}=C_{0,a}$ but the infimum \eqref{GNI} is never attained.
\end{itemize}
\end{theorem}
\begin{proof} Sobolev embedding and Lemma~\ref{EquiSobolev} immediately yield $C_{\alpha,a}<\infty$.  

First take $a\leq 0$. Following \cite{KillipOhPoVi2017}, we introduce the functional
\[
J_{a}(f)=\frac{\|f\|_{L^{2}}\| f\|^{\frac{3}{1+\alpha}}_{\dot{H}^{1}_{a}}\|f\|^{\frac{3\alpha}{1+\alpha}}_{L^{6}}}{\|f\|^{4}_{L^{4}}}
\] 
and take a minimizing sequence $\left\{f_{n}\right\}_{n}\in {H}^{1}_{a}$ with
\[
\lim_{n\to\infty}J_{a}(f_{n})=C^{-1}_{\alpha,a}.
\]

By Schwartz symmetrization (and the condition $a\leq 0$), we can assume that each $f_{n}$ is nonnegative and radially decreasing. By scaling, we may assume $\|f_{n}\|_{L^{2}}=1$ and  $\| f_{n}\|_{\dot{H}^{1}_{a}}=1$ for all $n\in \N$, so that $\{f_n\}$ is bounded in $H_{\text{rad}}^1$. Thus, there exists $f_{\ast}\in {H}^{1}_{a}$ such that (up to a subsequence) $f_{n}\to f_{\ast}$ strongly in $L^{4}$ and $f_{n}\rightharpoonup f_{\ast}$ weakly in ${H}^{1}_{a}$ and $L^{6}$ as $n\to \infty$.

We next observe that
\[\begin{split}
C^{-1}_{\alpha,a}=\lim_{n\to\infty}J_{a}(f_{n})=
\lim_{n\to\infty}\frac{\|f_{n}\|^{\frac{\alpha}{\alpha+1}}_{L^{2}}\|f_{n}\|^{\frac{3\alpha}{\alpha+1}}_{L^{6}}}{\|f_{n}\|^{4}_{L^{4}}}
\geq \lim_{n\to\infty} \frac{\|f_{n}\|^{\frac{4\alpha}{\alpha+1}}_{L^{4}}}{\|f_{n}\|^{4}_{L^{4}}}\\
=\lim_{n\to\infty}\|f_{n}\|^{-\frac{4}{\alpha+1}}_{L^{4}}=\|f_{\ast}\|^{-\frac{4}{\alpha+1}}_{L^{4}},
\end{split}\]
yielding $f_{\ast}\neq 0$. Moreover, as $f_{n}\rightharpoonup f_{\ast}$ weakly in $L^{6}$ and $f_{n}\to f_{\ast}$ strongly in $L^{4}$,
\[
C^{-1}_{\alpha,a}\leq J_{a}(f_{\ast})\leq \frac{\|f_{\ast}\|^{\frac{3\alpha}{\alpha+1}}_{L^{6}}}{\|f_{\ast}\|^{4}_{L^{4}}}\leq \lim_{n\to\infty}J_{a}(f_{n})=C^{-1}_{\alpha,a}.
\]
Thus $f_{\ast}$ is a minimizer, with $\|f_{\ast}\|_{L^{2}}=\| f_{\ast}\|_{\dot{H}^{1}_{a}}=1$ and $f_{n}\to f_{\ast}$ strongly in ${H}^{1}_{a}$.

In particular, $f_{\ast}$ is a solution to the Euler--Lagrange equation
\[
\left.\frac{d}{d\epsilon}J_{a}(f_{\ast}+\epsilon\varphi)\right|_{\epsilon=0}=0,
\quad \text{ for all $\varphi\in C^{\infty}_{0}(\R^{3})$},
\]
which implies that $f_*$ satisfies the elliptic equation 
\[
\L_{a}f_{\ast}+\tfrac{\alpha}{\| f_{\ast}  \|^{6}_{L^{6}}}f^{5}_{\ast}-
\tfrac{4(1+\alpha)}{3\| f_{\ast}  \|^{4}_{L^{4}}}f^{5}_{\ast}+
\tfrac{1+\alpha}{3}f_{\ast}=0.
\]
We now set $Q_{\alpha,a}(x)=\lambda^{-1}f_{\ast}(\mu^{-1}\,x)$, where
\[
\lambda^{2}=\tfrac{4(1+\alpha)}{3\alpha}\tfrac{\| f_{\ast}  \|^{6}_{L^{6}}}{\| f_{\ast}  \|^{4}_{L^{4}}}\quad \text{and}\quad 
\mu^{2}=\tfrac{16(1+\alpha)^{2}}{9\alpha}\tfrac{\| f_{\ast}  \|^{6}_{L^{6}}}{\| f_{\ast}  \|^{8}_{L^{4}}},
\]
so that $Q_{\alpha,a}$ solves \eqref{EllipNLS} with 
\[
\omega=\frac{3\alpha}{16(1+\alpha)}\frac{\| f_{\ast}  \|^{8}_{L^{4}}}{\| f_{\ast}  \|^{6}_{L^{6}}}>0.
\]

Using the H\"older inequality $\| f_{\ast}  \|^{4}_{L^{4}}\leq \| f_{\ast}  \|_{L^{2}}\| f_{\ast}  \|^{3}_{L^{6}}
=\| f_{\ast}  \|^{3}_{L^{6}}$ we deduce that $\omega\in (0,3\alpha/16(1+\alpha))$. Finally, it follows from 
straightforward calculations that
\[
\frac{\|Q_{\alpha, a}\|^{6}_{L^{6}}}{\|Q_{\alpha, a}\|^{2}_{\dot{H}^{1}_{a}}}
=\frac{\mu^{2}}{\lambda^{4}}\frac{\| f_{\ast}  \|^{6}_{L^{6}}}{\| f_{\ast}\|_{\dot{H}^{1}_{a}}}=\alpha,
\]
which completes the proof of part (i) of theorem.

We turn now to part (ii) and so assume $a>0$. Let us show that $J_{a}$ has no minimizer when $a>0$. Since $a>0$, it is clear that
\begin{equation}\label{NoAttai}
\begin{split}
\|f\|^{4}_{L^{4}}\leq 
C_{\alpha, 0}\|f\|_{L^{2}}\|f\|^{\frac{3}{1+\alpha}}_{\dot{H}^{1}}\|f\|^{\frac{3\alpha}{1+\alpha}}_{L^{6}}\\
<C_{\alpha, 0}\|f\|_{L^{2}}\|f\|^{\frac{3}{1+\alpha}}_{\dot{H}^{1}_{a}}\|f\|^{\frac{3\alpha}{1+\alpha}}_{L^{6}},
\end{split}
\end{equation}
for all $f\in {H}^{1}\setminus\left\{0\right\}$. This implies that $C_{\alpha, a}\leq C_{\alpha, 0}$. On the other hand, 
consider a sequence  $\left\{x_{n}\right\}_{n\in \N}$ such that $|x_{n}|\to \infty$. As
\[\lim_{n\to\infty}\|Q_{\alpha, 0}(\cdot-x_{n})\|_{\dot{H}^{1}_{a}}=\|Q_{\alpha, 0}\|_{\dot{H}^{1}}
\]
(cf. \eqref{Conver33} below), it follows that
\begin{equation}\label{JLimitN}
\lim_{n\to\infty} J_{a}(Q_{\alpha, 0}(\cdot-x_{n}))=J_{0}(Q_{\alpha, 0})=C^{-1}_{\alpha, 0},
\end{equation}
that is, $C_{\alpha, 0}\leq C_{\alpha, a}$. Therefore $C_{\alpha, 0}= C_{\alpha, a}$. 
Finally, \eqref{NoAttai} and \eqref{JLimitN} implies that the infimum  $C_{\alpha, a}$ is never attained. This completes the proof of theorem.
\end{proof}

\begin{remark}\label{Phoza}
If $\varphi\in {H}^{1}_{a}\setminus\left\{0\right\}$ satisfies the elliptic equation \eqref{EllipNLS} for some $\omega\in \C$, then the following Pohozaev identities hold:
\begin{align} \label{Pho11}
&\|\varphi\|^{2}_{\dot{H}^{1}_{a}}+\|\varphi\|^{6}_{L^{6}}-\|\varphi\|^{4}_{L^{4}}+\omega\|\varphi\|^{2}_{L^{2}}=0,\\
&\tfrac{1}{6}\|\varphi\|^{2}_{\dot{H}^{1}_{a}}+\tfrac{1}{6}\|\varphi\|^{6}_{L^{6}}-\tfrac{1}{4}\|\varphi\|^{4}_{L^{4}}
+\tfrac{\omega}{2}\|\varphi\|^{2}_{L^{2}}=0. \label{Pho22}
\end{align}
Indeed, to obtain \eqref{Pho11}, we multiply \eqref{EllipNLS} by $\bar{\varphi}$ and integrate over $\R^{3}$. Similarly, 
multiplying \eqref{EllipNLS} by $x\cdot\nabla \varphi$ and integrating leads to \eqref{Pho22}. 

Using these identities, we can deduce that 
\[
\omega\in(0, \tfrac{3}{16}).
\]
Indeed, combining \eqref{Pho11} and \eqref{Pho22} we obtain 
$\|\varphi\|^{4}_{L^{4}}=4\omega\|\varphi\|^{2}_{L^{2}}$. This implies that $\omega>0$. On the other hand, if $\omega\geq \frac{3}{16}$, then $\frac{1}{6}x^{6}-\frac{1}{4}x^{4}+\frac{\omega}{2}x^{2}\geq 0$ for all $x$. In this case \eqref{Pho22} yields the contradiction $\varphi\equiv 0$.
\end{remark}

\begin{corollary}[The sharp constant $C_{\alpha, a}$]\label{GNSharp}
For $a\leq 0$,  the sharp constant $C_{\alpha, a}$ in the $\alpha$-Gagliardo-Nirenberg-H\"older inequality \eqref{InvGN}
is  given by 
\begin{equation}\label{OptC}
C_{\alpha, a}=\frac{4(1+\alpha)}{3\alpha^{\frac{\alpha}{2(1+\alpha)}}}
\frac{\| Q_{\alpha,a}\|^{\frac{\alpha-1}{\alpha+1}}_{\dot{H}^{1}_{a}}}{\|Q_{\alpha,a}\|_{L^{2}}},
\end{equation}
where $Q_{\alpha,a}$ is the optimizer given in Theorem \ref{TheInverGN}(i).
\end{corollary}
\begin{proof}
Combining the Pohozaev identities \eqref{Pho11} and \eqref{Pho22} we obtain
\begin{equation}\label{Ne0}
\| Q_{\alpha,a}\|^{4}_{L^{4}}=\tfrac{4}{3}(1+\alpha)\|  Q_{\alpha,a}\|^{2}_{\dot{H}^{1}_{a}}.
\end{equation}
As ${\|Q_{\alpha, a}\|^{6}_{L^{6}}}={\alpha}{\|Q_{\alpha, a}\|^{2}_{\dot{H}^{1}_{a}}}$, it follows that
\begin{align*}
C_{\alpha, a}&=\frac{1}{J_{a}(Q_{\alpha,a})}=
\frac{\| Q_{\alpha,a}\|^{4}_{L^{4}}}{\|Q_{\alpha,a}\|_{L^{2}}\|Q_{\alpha,a}\|^{\frac{3}{1+\alpha}}_{\dot{H}^{1}_{a}}\|Q_{\alpha,a}\|_{L^{6}}^{\frac{3\alpha}{1+\alpha}}}=\frac{4(1+\alpha)}{3\alpha^{\frac{\alpha}{2(1+\alpha)}}}
\frac{\|Q_{\alpha,a}\|^{\frac{\alpha-1}{\alpha+1}}_{\dot{H}^{1}_{a}}}{\|Q_{\alpha,a}\|_{L^{2}}}.
\end{align*}
\end{proof}

\begin{remark}\label{Compa} We have $C^{-1}_{1, a}\leq C^{-1}_{1, 0}$ for any $a\in(-\frac{1}{4},0]$. In particular, from \eqref{OptC} we see that
\[
\| Q_{1,a}\|_{L^{2}}=\tfrac{8}{3}C^{-1}_{1, a}\leq \tfrac{8}{3}C^{-1}_{1, 0}=\| Q_{1,0}\|_{L^{2}}.
\]
\end{remark}
\subsection{Variational analysis}\label{DomainsC}
Throughout the rest of the paper, we let $Q_{1, a\wedge 0}$ denote an optimizer of  \eqref{GNI} given in Theorem \ref{TheInverGN}(i) with $\alpha=1$.  Noting that \eqref{Ne0} implies
$\| Q_{1, a\wedge 0}\|^{4}_{L^{4}}=\frac{8}{3}\|  Q_{1, a\wedge 0}\|^{2}_{\dot{H}^{1}_{a}}$ and recalling
${\|Q_{1, a\wedge 0}\|^{6}_{L^{6}}}={\|Q_{1, a\wedge 0}\|^{2}_{\dot{H}^{1}_{a}}}$, we observe that
\[
E_{a\wedge 0}(Q_{1, a\wedge 0})=\tfrac{1}{2}\|  Q_{1, a\wedge 0}\|^{2}_{\dot{H}^{1}_{a\wedge 0}}
-\tfrac{1}{4}\| Q_{1, a\wedge 0}\|^{4}_{L^{4}}+\tfrac{1}{6}\| Q_{1, a\wedge 0}\|^{6}_{L^{6}}=0.
\]

To begin the analysis, we define
\begin{align}\label{Vpfree}
d_{a}({m})&:=\inf\left\{E_{a\wedge 0}(f):\,\, f\in H^{1}(\R^{3}),\,\, M(f)=m  \right\},\\
m_{\ast}&:=\sup\left\{m>0: d_{a}({m})=0\right\}.\nonumber
\end{align}

We then have the following: 
\begin{proposition}\label{FiVP}
Let $a>-\frac{1}{4}$ and  $\alpha\in(0,\infty)$.
\begin{itemize}
\item[(i)] Assume $0\leq m\leq M(Q_{1, a\wedge 0})$. Then $d_{a}({m})=0$.
\item[(ii)] Assume $m> M(Q_{1, a\wedge 0})$. Then $d_{a}({m})<0$. In particular, $m_{\ast}=M(Q_{1, a\wedge 0})$. 
\item[(iii)] The infimum function $d_{a}:[0,\infty)\rightarrow \R$  is continuous, non-increasing and non-positive on $[0,\infty)$.
Moreover, if $m\geq M(Q_{1, a\wedge 0})$, then the variational problem \eqref{Vpfree} is well-defined and $d_{a}({m})=E_{a\wedge 0}(f_{\ast})$
for some $f_{\ast}\in H^{1}(\R^{3})$.
\end{itemize}
\end{proposition}
\begin{proof} For (i), we note that $d_{a}({m})\leq 0$ for $m\geq 0$. Indeed, the functions $g_{s}(x):=s^{\frac{3}{2}}g(sx)$ obey $M(g_{s})=M(g)$ and
\[E_{a\wedge 0}(g_{s})=\tfrac{s^{2}}{2}\|  g\|^{2}_{\dot{H}^{1}_{a\wedge 0}}
-\tfrac{s^{3}}{4}\| g  \|^{4}_{L^{4}}
+\tfrac{s^{6}}{6}\| g  \|^{6}_{L^{6}}\rightarrow 0 \quad \text{as $s\rightarrow0$},
\]
yielding $d_{a}(m)\leq 0$. Moreover, by the $(\alpha=1)$-Gagliardo-Nirenberg inequality \eqref{InvGNew} and \eqref{OptC} we have
\[\|f\|^{4}_{L^{4}}\leq 
\tfrac{8}{3}\( \tfrac{M(f)}{M(Q_{1, a\wedge 0})} \)^{\frac{1}{2}}\|f\|^{\frac{3}{2}}_{\dot{H}^{1}_{a\wedge 0}}\|f\|^{\frac{3}{2}}_{L^{6}}.  \]
Thus, by Young's inequality, we obtain
\begin{equation}\label{IIN}
\begin{split}
E_{a\wedge 0}(f)\geq \tfrac{1}{2}\|f\|^{2}_{\dot{H}^{1}_{a\wedge 0}}+\tfrac{1}{6}\|f\|^{6}_{L^{6}}-\tfrac{2}{3}
\( \tfrac{M(u)}{M(Q_{1, a\wedge 0})} \)^{\frac{1}{2}}\|f\|^{\frac{3}{2}}_{\dot{H}^{1}_{a\wedge 0}}\|f\|^{\frac{3}{2}}_{L^{6}}\\
\geq \left[1-\( \tfrac{M(f)}{M(Q_{1, a\wedge 0})} \)^{\frac{1}{2}} \right]\left[\tfrac{1}{2}\|f\|^{2}_{\dot{H}^{1}_{a\wedge 0}}+\tfrac{1}{6}\|f\|^{6}_{L^{6}}\right].
\end{split}
\end{equation}
This yields $E_{a\wedge 0}(f)\geq 0$ when $0\leq M(f)\leq M(Q_{1, a\wedge 0})$, which implies (i).

We turn to (ii). First note that $E_{a\wedge 0}(Q_{1, a\wedge 0})=0$ and \eqref{IIN} yield $d(M(Q_{1, a\wedge 0}))=0$. 
Next, suppose $m>M(Q_{1, a\wedge 0})$ and set 
\[
Q^{s}_{1, a\wedge 0}(x)=s^{-\frac{1}{2}}Q_{1, a\wedge 0}(s^{-1}x),\qtq{where} s^{2}=m/M(Q_{1, a\wedge 0}).
\]
Then $s>1$, $M(Q^{s}_{1, a\wedge 0})=m$ and 
\[E_{a\wedge 0}(Q^{s}_{1, a\wedge 0})=E_{a\wedge 0}(Q_{1, a\wedge 0})-\tfrac{\( s-1\)}{4}\|Q_{1, a\wedge 0}\|^{4}_{L^{4}_{x}}
=-\tfrac{\( s-1\)}{4}\|Q_{1, a\wedge 0}\|^{4}_{L^{4}_{x}}<0.
\]
Consequently, $d_{a}(m)<0$ when $m>M(Q_{1, a\wedge 0})$, which yields (ii).

Finally, we prove (iii). We first show that $d_{a}(m)$ is non-increasing for $m\geq0$. Given $0<m_{1}<m_{2}$, we choose
$f\in H^{1}(\R^{3})$ such that $M(f)=m_{1}$. We define 
$f^{s}(x):=s^{-\frac{1}{2}}f(s^{-1}x)$ with $s^{2}=m_{2}/m_{1}$. By definition, $s^{2}>1$, $M(f^{s})=m_{2}$, and
\begin{equation}\label{Contra}
E_{a\wedge 0}(f^{s})=E_{a\wedge 0}(f)-\tfrac{\( s-1\)}{4}\|f\|^{4}_{L^{4}_{x}}<E_{a\wedge 0}(f),
\end{equation}
yielding $d_{a}(m_{2})\leq d_{a}(m_{1})$.  We now show that the 
minimizer of $d_{a}(m)$ is achieved for all $m>M(Q_{1, a\wedge 0})$. Let $\left\{ f_{n} \right\}_{n\in \N}$
be a minimizing sequence for $d_{a}(m)$. Since
\begin{equation}\label{IdenE}
E_{a\wedge 0}(f_{n})+\tfrac{3}{32}M(f_{n})=\tfrac{1}{2}\|f_{n}\|^{2}_{\dot{H}^{1}_{a\wedge 0}}+
\tfrac{1}{6}\int_{\R^{3}}|f_{n}|^{2}\Big(|f_{n}|^{2}-\tfrac{3}{4}\Big)^{2}dx,
\end{equation}
it follows that the sequence $\left\{ f_{n} \right\}_{n\in \N}$ is bounded in ${H}^{1}$.
By Schwartz symmetrization, we can assume that $\left\{ f_{n} \right\}_{n\in \N}$ is radial for all $n$. Thus there exists $f_{\ast}$ such that (passing to a subsequence), we have $f_n$ converges weakly to $f_\ast$ in $H^1$ and $L^6$ and strongly in $L^4$. By weak lower-semicontinuity we see that
\[
E_{a\wedge 0}(f_{\ast})\leq d_{a}(m)<0 \quad \text{and}\quad  M(f_{\ast})\leq m.
\]
In particular $f_{\ast}\neq 0$. Moreover, if $M(f_{\ast})< m$, then the same argument given above shows that there exists $g\in {H}^{1}$ such that $M(g)=m$ and
\[
d_{a}(m)\leq E_{a\wedge 0}(g)<E_{a\wedge 0}(f_{\ast})\leq d_{a}(m)
\]
(cf. \eqref{Contra}), which is a contradiction. Thus we must have $M(f_{\ast})= m$ and $E_{a\wedge 0}(f_{\ast})= d_{a}(m)$.  On the other hand, as mentioned above, $d_{a}(m)$ is non-positive and non-increasing for $m\geq 0$. Finally, the continuity of $d_a(m)$ follows as in the proof of \cite[Theorem~4.1]{KillipOhPoVi2017}. \end{proof}

We now return to the variational problem defined in the introduction, namely, 
\begin{equation}\label{VirialCharac}
\EE_{a}(m):=\inf\left\{E_{a\wedge 0}(f): f\in H^{1}(\R^{3}),\, M(f)=m \,\, \mbox{and}\,\, V_{a\wedge 0}(f)=0 \right\}, 
\end{equation}
where $V_{a}$ is the virial functional
\begin{equation}\label{VirialF}
V_{a}(f)=\| f\|^{2}_{\dot{H}^{1}_{a}}+\|f\|^{6}_{L^{6}}
-\tfrac{3}{4}\|f\|^{4}_{L^{4}}.
\end{equation}
By definition, $\EE_{a}(m)=\infty$ when the set $\left\{M(u)=m \,\, \mbox{and}\,\, V_{a\wedge 0}(f)=0 \right\}$ is empty. We also recall the region $\K_{a}\subset \R^{2}$ given by
\[
\K_{a}:= \left\{(m,e): 0<m<M(Q_{1, a\wedge 0}) \,\, \mbox{and} \,\, 0<e<\EE_{a}(m)\right\}.
\]
Finally, we set $S_{a}(x):=\frac{1}{\sqrt{2}}Q_{1, a\wedge 0}(\frac{\sqrt{3}}{2}x)$.

\begin{theorem}\label{Vscatte2}
Let $a>-\frac{1}{4}$ and $f\in {H}^{1}$. The following statements hold.
\begin{itemize}
\item[(i)] If  $(M(f), E_{a}(f))\in \K_{a}$, then $V_{a\wedge 0}(f)>0$.
\item[(ii)] If $0<m<M(S_{a})$, then  $\EE_{a}(m)=\infty$.
\item[(iii)] If $M(S_{a})\leq m<M(Q_{1, a\wedge 0})$, then  $0<\EE_{a}(m)<\infty$.
\item[(iv)] If $m\geq M(Q_{1, a\wedge 0})$ then $\EE_{a}(m)=d_{a}(m)$. In particular, $\EE_{a}(M(Q_{1, a\wedge 0}))=0$.  Furthermore, the infimum $\EE_{a}(m)$ is achieved and the infimum function $\EE_{a}(m)$ is strictly decreasing  and lower semicontinuous. 
\end{itemize}
\end{theorem}

The proof relies on the following lemma, whose proof we omit, as it is essentially the same as that of \cite[Lemmas 5.3 and 5.4]{KillipOhPoVi2017}.

\begin{lemma}\label{KL1}
Let $m>0$ and $f\in H^{1}(\R^{3})\setminus\left\{0\right\}$. Then:
\begin{itemize}
\item[a.] Assume that $V_{a\wedge 0}(f)<0$. 
Writing $f^{s}(x):=s^{\frac{3}{2}}f(sx)$, there exists $s>1$ such that $V_{a\wedge 0}(f^{s})=0$ and  $E_{a\wedge 0}(f^{s})<E_{a\wedge 0}(f)$. 
\item[b.]  If $f$ satisfies $0<M(f)<m$ and $V_{a\wedge 0}(f)=0$, then there exists $f_{\ast}\in H^{1}(\R^{3})$ such that
\begin{equation}\label{IMporC}
\begin{split}
M(f_{\ast})=m,\quad E_{a\wedge 0}(f_{\ast})\leq E_{a\wedge 0}(f)-
\( \tfrac{m-M(f)}{6M(f)}\)\|f\|^{2}_{\dot{H}^{1}_{a}},
\quad \mbox{and}\quad V_{a\wedge 0}(f_{\ast})=0.
\end{split}
\end{equation}
\end{itemize}
\end{lemma}

\begin{proof}[{Proof of Theorem \ref{Vscatte2}}]\text{ }

(i) Consider $f\in H^{1}(\R^{3})$ such that $(M(f), E_{a}(f))\in \K_{a}$.
 By definition of the set $\K_{a}$, it is clear that $V_{a\wedge 0}(f)\neq 0$. Suppose that $V_{a\wedge 0}(f)<0$. 
From Lemma~\ref{KL1}(a), we infer that there exists $s>1$
such that  
\[
M(f^{s})=M(f),\quad V_{a\wedge 0}(f^{s})=0,\qtq{and}E_{a\wedge 0}(f^{s})<E_{a\wedge 0}(f).
\]
In this case, by the definition of $\EE_{a}(m)$, we see that 
$\EE_{a}(m)<E_{a\wedge 0}(f)\leq E_{a}(f)$, which is impossible since $(M(f), E_{a}(f))\in \K_{a}$. 

(ii) It suffices to show that no function obeys the constraints
\[
f\in H^{1}(\R^{3}),\quad M(f)=m,\qtq{and} V_{a\wedge 0}(f)=0
\]
if $0<m<M(S_{a})$. To do this, we will express the sharp constant $C_{1,a\wedge 0}$ in terms of the function $S_a$.  To this end, we first define
\begin{equation}\label{DefSca}
\varphi^{r,b}(x):=rQ_{1, a\wedge 0}(bx), \quad r>0,\quad b>0. 
\end{equation}
Direct calculations show that
\begin{equation}\label{ProScaling}
\begin{split}
\| \varphi^{r,b}\|^{2}_{\dot{H}^{1}_{a}}&=r^{2}b^{-1}\| Q_{1, a\wedge 0}\|^{2}_{\dot{H}^{1}_{a}}, 
\quad \| \varphi^{r,b}\|^{4}_{L^{4}}=r^{4}b^{-3}\| Q_{1, a\wedge 0}\|^{4}_{L^{4}},\\
\|\varphi^{r,b}\|^{2}_{L^{2}}&=r^{2}b^{-3}\| Q_{1, a\wedge 0}\|^{2}_{L^{2}}, \quad
\| \varphi^{r,b}\|^{6}_{L^{6}}=r^{6}b^{-3}\| Q_{1, a\wedge 0}\|^{6}_{L^{6}}.
\end{split}
\end{equation}
In particular, from \eqref{Ne0}, we obtain 
\[V_{a\wedge 0}(\varphi^{r,b})=\(\tfrac{r^{2}}{b}+ \tfrac{r^{6}}{b^{3}}-2\tfrac{r^{4}}{b^{3}}\)\| Q_{1, a\wedge 0}\|^{2}_{\dot{H}^{1}_{a}}.\]

Now with $r=\frac{1}{\sqrt{2}}$ and $b=\frac{\sqrt{3}}{2}$ we obtain
\begin{equation}\label{ProScaling11}
M(S_{a})=\tfrac{4}{3\sqrt{3}}M(Q_{1, a\wedge 0}) 
\quad \text{and}\quad 
\quad V_{a\wedge 0}(S_{a})=0.
\end{equation}
We also note that
\begin{equation}\label{Idenmass}
\|S_{a}\|^{6}_{L^{6}}=\tfrac{1}{3}\| S_{a}\|^{2}_{\dot{H}^{1}_{a\wedge 0}}
\quad \text{and}\quad 
\|S_{a}\|^{4}_{L^{4}}=\tfrac{4^{2}}{3^{3}}\| S_{a}\|^{2}_{\dot{H}^{1}_{a\wedge 0}}.
\end{equation}
Now, since the functional $J_{a}$ (from proof of Theorem \ref{TheInverGN})  is invariant under the scaling \eqref{DefSca}, it follows that $S_{a}$ is a minimizer for the variational problem \eqref{GNI}. Thus, by \eqref{Idenmass} we obtain

\begin{equation}\label{EquaSA}
\begin{split}
C_{1, a\wedge 0}&=\frac{\|S_{a}\|^{4}_{L^{4}}}
{\|S_{a}\|_{L^{2}} \| S_{a}\|_{\dot{H}^{1}_{a\wedge 0}}^{\frac{3}{2}}\|S_{a}\|^{\frac{3}{2}}_{L^{6}}}
=\frac{4^{2}}{3^{2}} \frac{\| S_{a}\|^{2}_{\dot{H}^{1}_{a\wedge 0}} 3^{\frac{1}{4}}}
{\|S_{a}\|_{L^{2}}\| S_{a}\|_{\dot{H}^{1}_{a\wedge 0}}^{\frac{1}{2}}\| S_{a}\|_{\dot{H}^{1}_{a\wedge 0}}^{\frac{3}{2}} }\\
&= \frac{4^{2}}{3^{2}}\(\frac{3^{\frac{1}{4}}}{\|S_{a}\|_{L^{2}}}\). 
\end{split}
\end{equation}
By using Young's inequality, we have
\begin{equation}\label{Oldinequ}
\begin{split}
\|f\|^{4}_{L^{4}}\leq
C_{1,a\wedge 0}\|f\|_{L^{2}}\|f\|^{\frac{3}{2}}_{\dot{H}^{1}_{a\wedge 0}}\|f\|^{\frac{3}{2}}_{L^{6}}
&\leq
\frac{4^{2}}{3^{2}}\(\frac{3^{\frac{1}{4}}}{\|S_{a}\|_{L^{2}}}\)\|f\|_{L^{2}}\|f\|^{\frac{3}{2}}_{\dot{H}^{1}_{a\wedge 0}}\|f\|^{\frac{3}{2}}_{L^{6}}\\
&\leq \frac{4}{3}\frac{\|f\|_{L^{2}}}{\|S_{a}\|_{L^{2}}}\bigl[\|f\|^{2}_{_{\dot{H}^{1}_{a\wedge 0}}}+\|f\|^{6}_{L^{6}}\bigr].
\end{split}
\end{equation}
From this we infer that 
\[
V_{a\wedge 0}(f)>0\qtq{whenever}0<M(f)<M(S_{a})
\] 
as claimed.

(iii) Assume  $M(S_{a})\leq m<M(Q_{1, a\wedge 0})$.  Let $f\in \left\{M(u)=m \,\, \mbox{and}\,\, V_{a\wedge 0}(u)=0 \right\}$.
Since $V_{a\wedge 0}(f)=0$, by the $(\alpha=1)$-Gagliardo-Nirenberg-H\"older inequality \eqref{InvGNew}, it follows that
\[\|f\|^{2}_{\dot{H}^{1}_{a\wedge 0}}=\tfrac{3}{4}\|f\|^{4}_{L^{4}}-\|f\|^{6}_{L^{6}}\lesssim
 \|f\|^{3}_{\dot{H}^{1}_{a\wedge 0}}\| f\|_{L^{2}}, 
\]
and hence
\begin{equation}\label{InequGra}
\|f\|_{\dot{H}^{1}_{a\wedge 0}}\| f\|_{L^{2}}\gtrsim 1.
\end{equation}

Next, estimating as we did for \eqref{IIN}, we obtain
\begin{equation}\label{PositiveE}
\begin{split}
E_{a\wedge 0}(f)&\geq \tfrac{1}{2}\|f\|^{2}_{\dot{H}^{1}_{a\wedge 0}}+\tfrac{1}{6}\|f\|^{6}_{L^{6}}-
\tfrac{2}{3}\(\tfrac{M(f)}{M(Q_{1, a\wedge 0})}\)^{\frac{1}{2}}\|f\|^{\frac{3}{2}}_{\dot{H}^{1}_{a\wedge 0}}\| f\|^{\frac{3}{2}}_{L^{6}}\\
&\geq \bigl[1-\(\tfrac{M(f)}{M(Q_{1, a\wedge 0})}\)^{\frac{1}{2}}\bigr]\left\{\tfrac{1}{2}\|f\|^{2}_{\dot{H}^{1}_{a\wedge 0}}+\tfrac{1}{6}\|f\|^{6}_{L^{6}}\right\}\\
& \gtrsim\bigl[1-\(\tfrac{M(f)}{M(Q_{1, a\wedge 0})}\)^{\frac{1}{2}}\bigr][M(f)]^{-1}.
\end{split}
\end{equation}
Taking the infimum on the set $\left\{M(u)=m \,\, \mbox{and}\,\, V_{a\wedge 0}(u)=0 \right\}$ we infer that $\EE_{a}(m)>0$.
Finally,  Lemma \ref{KL1}(b) (with $f=S_{a}$) implies that  $\EE_{a}(m)\leq E_{a\wedge 0}(S_{a})$ if $M(S_{a})\leq M(f)$.
This completes the proof of (iii).

(iv) We first show show that $\EE_{a}(m)=d_{a}(m)$ for $m\geq M(Q_{1, a\wedge 0})$.  On one hand, it is clear that $d_{a}(m)\leq \EE_{a}(m)$. On the other hand, by Proposition \ref{FiVP}  we know that there exists
$f_{\ast}\in {H}^{1}$ with $E_{a\wedge 0}(f_{\ast})=d_{a}(m)$ and $M(f_{\ast})=m$.  We first observe that $V_{a\wedge 0}(f_{\ast})=0$. Indeed, $f_{\ast}$ satisfies the elliptic equation
\eqref{EllipNLS} for some $\omega>0$, which implies by \eqref{Pho11} and \eqref{Pho22} that $V_{a\wedge 0}(f_{\ast})=0$ holds. Therefore, by definition, 
\[
\EE_{a}(m)\leq E_{a\wedge 0}(f_{\ast})=d_{a}(m).
\]

Next we will show that $\EE_{a}(m)$ is strictly decreasing on $[M(S_{a}), M(Q_{1, a\wedge 0})]$. Indeed, 
consider $m_{{2}}<m_{1}$ such that $m_{{2}}$, $m_{1}\in [M(S_{a}), M(Q_{1, a\wedge 0})]$. Moreover,  let  $\left\{f_{n}\right\}_{n\in \N}$ be a minimizing sequence for $\EE_{a}(m_{2})$. Then we have $M(f_{n})=m_{2}$,
$V_{a\wedge 0}(f_{n})=0$ and $E_{a\wedge 0}(f_{n})\rightarrow \EE_{a}(m_{2})$. Since $V_{a\wedge 0}(f_{n})=0$, 
 applying the same argument as above (see \eqref{InequGra}) we see that there exists a constant $C>0$ (independent of $n$) such that $\|f_{n}\|_{\dot{H}^{1}_{a}}\geq C/m_{2}$.  Using Lemma \ref{KL1}(b) we obtain a sequence $\left\{g_{n}\right\}_{n\in \N}$ such that $V_{a\wedge 0}(g_{n})=0$, $M(g_{n})=m_{1}$  and 
\[E_{a\wedge 0}(g_{n})\leq E_{a\wedge 0}(f_{n})-C\frac{m_{1}-m_{2}}{6m^{2}_{2}}.\]
Since $m_{1}>m_{2}$, by the definition of  $\EE_{a}(m)$ we get $\EE(m_{1})<\EE(m_{2})$.

On the other hand, by using the fact that $\EE_{a}(m)$ is strictly decreasing on $[M(S_{a}), M(Q_{1, a\wedge 0})]$,
Lemma \ref{KL1}(b) and applying the argument in \cite[Theorem~5.2]{KillipOhPoVi2017}, we can show that that the minimization problem $\EE_{a}(m)$ is achieved for $m\in [M(S_{a}), M(Q_{1, a\wedge 0})]$.
Finally, the proof of the lower semicontinuity of $\EE_{a}(m)$ is also similar to that of \cite[Theorem 5.2]{KillipOhPoVi2017}, and so we omit the
details. This completes the proof of theorem.\end{proof}

\begin{corollary}[Comparison of thresholds]\label{Compa22}
Let $a>-\frac{1}{4}$. Then we have the inclusion $\K_{a}\subseteq \K_{0}$.
\end{corollary}
\begin{proof}
If $a\geq 0$, then $\K_{a}=\K_{0}$ by definition. If $a<0$, then by Remark~\ref{Compa} and \eqref{EquaSA} we have
\begin{equation}\label{MS}
M(Q_{1,a})\leq M(Q_{1,0})\quad \text{and}\quad M(S_{a})\leq M(S_{0}).
\end{equation}
Moreover, it is clear that $\EE_{a}(m)\leq \EE_{0}(m)$ when $m\in (0, M(S_{0}))$. Indeed, by Theorem~\ref{Vscatte2}(ii), we have $\EE_{0}(m)=\infty$. 

Next, assume $m\in [M(S_{0}), M(Q_{1,a})]$. Using \eqref{MS} we infer that $\EE_{0}(m)<\infty$. Then  
there exists $f\in H^{1}(\R^{3})$ such that $E_{0}(f)=\EE_{0}(m)$,  $M(f)=m$ and $V_{0}(f)=0$. Since $a<0$, it follows that
\begin{equation}\label{IneqE11}
E_{a}(f)<E_{0}(f)\quad \text{and}\quad V_{a}(f)<V_{0}(f)=0.
\end{equation}
Thus, by Lemma \ref{KL1} we see that there exists $f^{\ast}$ such that 
\begin{equation}\label{IneqE22}
E_{a}(f^{\ast})\leq E_{a}(f),\quad M(f^{\ast})=m \quad  \text{and}\quad V_{a}(f^{\ast})=0.
\end{equation}
Combining \eqref{IneqE11} and \eqref{IneqE22} we get 
\[\EE_{a}(m)\leq E_{a}(f^{\ast}) \leq E_{0}(f)=\EE_{0}(m).
\]
for all  $m\in [M(S_{0}), M(Q_{1,a})]$. This proves the result.
\end{proof}

\begin{remark}\label{IneE}
Let $a\in (-\frac{1}{4}, 0)$. By \eqref{IneqE22}, we can show that $\EE_{a}(m)<\EE_{0}(m)$ if $m\in [M(S_{a}), M(Q_{1,a})]$.  
\end{remark}

We next introduce the functional that will be used to set up the induction scheme for Theorem~\ref{MainTheorem}. For $a>-\frac{1}{4}+\frac{1}{25}$, we define 
\[\Omega_{a}:=\left\{(m,e)\in \R^{2}: m\geq M(S_{a})\quad \mbox{and}\quad e\geq \EE_{a}(m)\right\}\]
and let $\F: H^{1}(\R^{3})\rightarrow [0,\infty)$ be the continuous function
\begin{equation}\label{F-functional}
\F(f):=
\begin{cases} 
E_{a}(f)+\frac{M(f)+E_{a}(f)}{\mbox{dist}\((M(f),E_{a}(f)),\Omega_{a}\)} &  (M(f),E_{a}(f))\notin \Omega_{a}\\
\infty & \mbox{otherwise}.
\end{cases} 
 \end{equation}
Note that if $u$ solves \eqref{NLS}, then $\F(u(t))=\F(u|_{t=0})$ for all $t\in \R$.
Moreover, $\Omega_{a}=\Omega_{0}$ for $a\geq 0$.

\begin{lemma}\label{FunctionF}
Let $f\in H^{1}(\R^{3})$. The function $\F$ satisfies the following properties:
\begin{itemize}
\item[(i)] $0<\F(f)<\infty$ if and only if $(M(f),E_{a}(f))\in \K_{a}$. Moreover, $\F(f)=0$ if and only if $f\equiv 0$.  
\item[(ii)] If $0<\F(f)<\infty$, then  $V_{a}(f)>0$, where $V_{a}$ is as in \eqref{VirialF}.
\item[(iii)] If $M(f_{1})\leq M(f_{2})$ and $E_{a}(f_{1})\leq E_{a}(f_{2})$, then $\F(f_{1})\leq \F(f_{2})$.
\item[(iv)] Let $\F_{0}\in (0, \infty)$. Assume that $\F(f)\leq \F_{0}$, then we have
\begin{equation}\label{Enl}
\|f\|^{2}_{\dot{H}^{1}_{a}} \sim_{\F_{0}} E_{a}(f), \quad \mbox{and} \quad \|f\|^{2}_{{H}^{1}_{a}} \sim_{\F_{0}} 
E_{a}(f)+M(f)\sim_{\F_{0}} \F(f). 
\end{equation}
\item[(v)] Consider $\left\{f_{n}\right\}_{n\in\N}\subset {H}^{1}(\R^{3})$. If $M(f_{n})\leq M_{\ast}$, $E_{a}(f_{n})\leq E_{\ast}$, and $\F(f_{n})\rightarrow \F(M_{\ast}, E_{\ast})$, then 
$(M(f_{n}), E_{a}(f_{n}))\rightarrow (M_{\ast}, E_{\ast})$.
\end{itemize}
\end{lemma}
\begin{proof}
(i) Suppose that $\F(f)<\infty$. Then by definition $(M(f),E_{a}(f))\notin \Omega_{a}$. 
But then $M(f)<M(Q_{1, a\wedge 0})$ and $E_{a}(f)<\EE_{a}(m)$. We will show that $E_{a}(f)\geq 0$. 
Indeed, from inequality \eqref{PositiveE} we get
\[
E_{a}(f)\geq E_{a\wedge 0}(f)\geq \bigl[1-\(\tfrac{M(f)}{M(Q_{1, a\wedge 0})}\)^{\frac{1}{2}}\bigr]\left\{\tfrac{1}{2}
\|f\|^{2}_{\dot{H}^{1}_{a\wedge 0}}+\tfrac{1}{6}\|f\|^{6}_{L^{6}}\right\}\geq0.
\]
Therefore, $e=E_{a}(f)\geq 0$, and  $\F(f)=0$ if and only if $f\equiv 0$.

(ii) Since $0<\F(f)<\infty$, it follows from (i) that $(M(f),E_{a}(f))\in \K_{a}$. Thus from Theorem~\ref{Vscatte2} we
obtain that
\[
V_{a}(f)\geq V_{a\wedge 0}(f)>0.
\]

(iii) Assume $M(f_{1})\leq M(f_{2})$ and $E_{a}(f_{1})\leq E_{a}(f_{2})$. From Theorem~\ref{Vscatte2} (monotonicity of $\EE_{a}(m)$) we deduce
\begin{equation}\label{Inedis}
\mbox{dist}\( (M(f_{1}), E_{a}(f_{1})), \Omega_{a}\)\geq \mbox{dist}\( (M(f_{2}), E_{a}(f_{2})), \Omega_{a}\).
\end{equation}
Then, by definition of $\F$, we obtain $\F(f_{1})\leq \F(f_{2})$.

(iv) Suppose that $\F(f)\leq \F_{0}$ with $0<\F_{0}<\infty$.  Item (i) implies that 
\[
M(f)<M(Q_{1, a\wedge 0})\qtq{and}0<E_{a}(f)<\EE_{a}(M(f)).
\]
Now we observe that $(M(Q_{1, a\wedge 0}), E_{a}(f))\in \Omega_{a}$. Indeed, $\EE_{a}(m)$ is monotone decreasing with $\EE_{a}(M(Q_{1, a\wedge 0}))=0$. Therefore,
\begin{align*}
  \mbox{dist}\( (M(f), E_{a}(f)), \Omega_{a}\)&\leq \mbox{dist}\( (M(f), E_{a}(f)), (M(Q_{1, a\wedge 0}), E_{a}(f))\)\\
&=M(Q_{1, a\wedge 0})-M(f).
\end{align*}
In particular,
\begin{equation}\label{IneFM}
 \F(f)\geq \tfrac{M(f)}{M(Q_{1, a\wedge 0})-M(f)}.
\end{equation}
As $1-\sqrt{x}\geq \frac{1-x}{2}$ for $x\in[0,1]$, we deduce from \eqref{IneFM} that
\[ 
1-\sqrt{\tfrac{M(f)}{M(Q_{1, a\wedge 0})}}\geq \tfrac{1}{2\F(u)+2},
\]
where we have used that $M(f)<M(Q_{1, a\wedge 0})$. Then by \eqref{PositiveE} we see that
\begin{equation}\label{FsimH}
\begin{aligned}
\F(f)\geq E_{a}(f) \geq E_{a\wedge 0}(f)&\geq  
 \tfrac{1}{2\F(f)+2}\left[\tfrac{1}{2}\|f\|^{2}_{\dot{H}^{1}_{a\wedge 0}}+\tfrac{1}{6}\|f\|^{6}_{L^{6}}\right] \\
&\geq\tfrac{1}{4\F(f)+4}\|f\|^{2}_{\dot{H}^{1}_{a\wedge 0}}.
\end{aligned}
\end{equation}
Moreover, by Sobolev embedding and the equivalence of Sobolev norms we have
\[E_{a}(f)= \tfrac{1}{2}\|f\|^{2}_{\dot{H}^{1}_{a}}+\tfrac{1}{6}\|f\|^{6}_{L^{6}}-\tfrac{1}{4}\|f\|^{4}_{L^{4}}
\lesssim  \|f\|^{2}_{\dot{H}^{1}_{a}} \(1+\F(f)^{2}\).\]
Combining this inequality with \eqref{FsimH}  we obtain 
\[
E_{a}(f)\sim \|f\|^{2}_{\dot{H}^{1}_{a}}
\] 
for every $f$ such that
$\F(f)\leq \F_{0}$. In particular, we deduce that $E_{a}(f)+M(f)\sim \|f\|^{2}_{{H}^{1}_{a}}$.

To complete the proof of (iv), we need to show that $E_{a}(f)+M(f) \sim \F(f)$. To this end, note that if $\frac{4 M(f)}{M(Q_{1, a\wedge 0})}\geq 1$, then, recalling $\F(f)\leq \F_{0}$ and $E_{a}(f)\geq 0$, it follows that 
\[ \F(f)\leq \(\tfrac{4\F_{0}}{M(Q_{1, a\wedge 0})}\) M(f)+E_{a}(f).\]
 On the other hand, if $\frac{4 M(f)}{M(Q_{1, a\wedge 0})}<1$, we have that $M(f)<M(S_{a})$ and therefore
\[ \mbox{dist}\( (M(f), E_{a}(f)), \Omega_{a}\)\geq M(S_{a})-M(f)
=\(\tfrac{4}{3\sqrt{3}}-\tfrac{1}{4}\)M(Q_{1, a\wedge 0})\geq \tfrac{1}{2}M(Q_{1, a\wedge 0}).
\]
By definition of $\F$, we obtain
\[
\F(f)\leq (1+2[M(Q_{1, a\wedge 0})]^{-1})E_{a}(f)+(2[M(Q_{1, a\wedge 0})]^{-1})M(f).
\]
Finally, combining \eqref{IneFM} and \eqref{FsimH} we obtain
\[E_{a}(f)+M(f)\lesssim_{\F_{0}} [1+M(Q_{1, a\wedge 0})]\F(f).\]
Item (v) is now an immediate consequence of the inequality \eqref{Inedis} and the definition of $\F$. This completes the proof of lemma.
\end{proof}

\section{Construction of minimal blowup solutions}\label{Sec:CMBS}

The goal of this section is to prove that if Theorem~\ref{MainTheorem} fails, then we may construct a blowup solution with mass-energy in the region $\mathcal{K}_a$ that is `minimal' in a suitable sense and obeys certain compactness properties.  In the next section, we will utilize a localized virial argument to preclude the possibility of such a solution, thus establishing Theorem~\ref{MainTheorem}. 

\subsection{Linear profile decomposition}\label{Sec:LinearProfile}

We first need a linear profile decomposition associated to the propagator $e^{-it\L_a}$ and adapted to the cubic-quintic problem.  In fact, the result follows by combining the techniques of \cite{KillipMiaVisanZhangZheng, KillipMurphyVisanZheng2017}, which developed concentration-compactness tools to address the NLS with inverse-square potential with either pure cubic or pure quintic nonlinearity, with those of \cite{KillipOhPoVi2017}, which developed concentration-compactness tools adapted to the cubic-quintic problem without potential.  Thus, we will focus on stating the main results and providing suitable references to the analogous results in the references just mentioned. 

First, given a sequence $\left\{x_{n}\right\}\subset \R^{3}$, we define
\begin{equation}\label{DefOperator}
\L^{n}_{a}:=-\Delta+\tfrac{a}{|x+x_{n}|^{2}}
\quad \text{and}\quad 
\L^{\infty}_{a}:=
\begin{cases} 
-\Delta+\tfrac{a}{|x+x_{\infty}|^{2}} & \text{if $x_{n}\to x_{\infty}\in \R^{3}$},\\
-\Delta & \text{if $|x_{n}|\to \infty$}.
\end{cases} 
\end{equation}
In particular, $\L_{a}[\phi(x-x_{n})]=[\L^{n}_{a}\phi](x-x_{n})$, and for any $x_{n}\in \R^{3}$ and $N_{n}>0$,
\[
N^{\frac{1}{2}}_{n} e^{-it\L_{a}}[\phi(N_{n}x-x_{n})]=N^{\frac{1}{2}}_{n}[e^{-iN^{2}_{n}t\L^{n}_{a}}\phi](N_{n}x-y_{n}).
\]

The proof requires several results related to the convergence of the operator $\L_a^n$ to $\L_a^\infty$, all of which we import from \cite{KillipMiaVisanZhangZheng, KillipMurphyVisanZheng2017}: 
\begin{lemma}\label{ConverOpera}
Fix $a>-\frac{1}{4}$. 
\begin{itemize}
\item If $t_{n}\to t_{\infty}\in \R$ and $\left\{x_{n}\right\}\subset \R^{3}$ satisfies $x_{n}\to x_{\infty}$ or $|x_{n}|\to \infty$, then
\begin{align}\label{Conver11}
	&\lim_{n\to\infty}\|\L^{n}_{a}\psi-\L^{\infty}_{a}\psi\|_{\dot{H}^{-1}_{x}}=0 \quad \text{for all $\psi\in \dot{H}^{1}_{x}$},\\
	\label{Conver22}
	&\lim_{n\to\infty}\|(e^{-it_{n}\L^{n}_{a}}-e^{-it_{\infty}\L^{\infty}_{a}})\psi\|_{\dot{H}^{-1}_{x}}=0
	\quad \text{for all $\psi\in \dot{H}^{-1}_{x}$},\\
	\label{Conver33}
	&\lim_{n\to\infty}\|[\sqrt{\L^{n}_{a}}-\sqrt{\L^{\infty}_{a}}]\psi\|_{L^{2}_{x}}=0
	\quad \text{for all $\psi\in \dot{H}^{1}_{x}$}.
\end{align}
If $\frac{2}{q}+\frac{3}{r}=\frac{3}{2}$ with $2<q\leq \infty$, then we have
\begin{equation}\label{Conver44}
\lim_{n\to\infty}\|(e^{-it\L^{n}_{a}}-e^{-it\L^{\infty}_{a}})\psi\|_{L^{q}_{t}L^{r}_{x}(\R\times\R^{3})}=0
	\quad \text{for all $\psi\in L^{2}_{x}$}.
\end{equation}
Finally, if $x_{\infty}\neq 0$, then for any $t>0$,
\begin{equation}\label{Conver55}
\lim_{n\to\infty}\|(e^{-it\L^{n}_{a}}-e^{-it\L^{\infty}_{a}})\delta_{0}\|_{H^{-1}}=0.
\end{equation}
\item Given $\psi\in \dot{H}^{1}_{x}$, $t_{n}\to \pm\infty$ and any sequence 
$\left\{x_{n}\right\}\subset \R^{3}$, we have
\begin{equation}\label{DecayingES1}
\lim_{n\to\infty}\|e^{-it_{n}\L^{n}_{a}}\psi\|_{L^{6}_{x}}=0.
\end{equation}
Moreover, if $\psi\in {H}^{1}_{x}$, then
\begin{equation}\label{DecayingES2}
\lim_{n\to\infty}\|e^{-it_{n}\L^{n}_{a}}\psi\|_{L^{4}_{x}}=0.
\end{equation}
\item Finally, fix $a>-\frac{1}{4}+\frac{1}{25}$. Then for any sequence $\{x_n\}$,
\begin{equation}\label{Convercritical}
\lim_{n\to\infty}\|(e^{-it\L^{n}_{a}}-e^{-it\L^{\infty}_{a}})\psi\|_{L^{10}_{t,x}(\R\times\R^{3})}=0
	\quad \text{for all $\psi\in \dot{H}^{1}_{x}$}.
\end{equation}
\end{itemize}
\end{lemma}

The linear profile decomposition is stated as follows: 

\begin{theorem}[Linear profile decomposition]\label{LinearProfi}
Let $\left\{f_{n}\right\}$ be a bounded sequence in $H^{1}(\R^{3})$. Then, up to subsequence, there exist
 $J^{\ast}\in \left\{0,1,2,\ldots\right\}\cup\left\{\infty\right\}$, non-zero profiles 
$\{\phi^{j}\}^{J^{\ast}}_{j=1}\subset \dot{H}^{1}(\R^{3})$ and parameters
\[
\left\{\lambda^{j}_{n}\right\}_{n\in \N}\subset (0,1],\quad 
\left\{t^{j}_{n}\right\}_{n\in \N}\subset \R
\quad \text{and}\quad 
\left\{x^{j}_{n}\right\}_{n\in \N}\subset \R^{3}
\]
so that for each finite $1\leq J\leq J^{\ast}$, we have the decomposition
\begin{equation}\label{Dcom}
f_{n}=\sum^{J}_{j=1}\phi_{n}^{j}+W^{J}_{n},
\end{equation}
where
\begin{equation}\label{fucti}
\phi_{n}^{j}(x):=
\begin{cases} 
[e^{-it^{j}_{n}\L^{n_{j}}_{a}}\phi^{j}](x-x^{j}_{n}), &\mbox{if $\lambda^{j}_{n}\equiv 1$},\\
(\lambda^{j}_{n})^{-\frac{1}{2}}[e^{-it^{j}_{n}\L^{n_{j}}_{a}} P^{a}_{\geq (\lambda^{j}_{n})^{\theta}}\phi^{j}]\( \frac{x-x^{j}_{n}}{\lambda^{j}_{n}}\),
&\mbox{if $\lambda^{j}_{n}\rightarrow 0$},
\end{cases}
\end{equation}
for some $0<\theta<1$ (with $\L^{n_{j}}_{a}$ as in \eqref{DefOperator} corresponding to sequence 
$\{\tfrac{x^{j}_{n}}{\lambda^{j}_{n}}\}$), satisfying
\begin{itemize}
	\item $\lambda^{j}_{n}\equiv 1$ or $\lambda^{j}_{n}\rightarrow 0$ and $t^{j}_{n}\equiv 0$ or $t^{j}_{n}\rightarrow\pm\infty$,
	\item if $\lambda^{j}_{n}\equiv 1$ then $\left\{\phi^{j}\right\}^{J^{\ast}}_{j=1}\subset L_{x}^{2}(\R^{3})$
\end{itemize}
for each $j$.  Furthermore, we have:
\begin{itemize}[leftmargin=5mm]
	\item Smallness of the reminder: 
	\begin{equation}\label{Reminder}
\lim_{J\to J^{\ast}}\limsup_{n\rightarrow\infty}\|e^{-it\L_{a}}W^{J}_{n}\|_{L^{10}_{t,x}(\R\times\R^{3})}=0.
   \end{equation}
		\item Weak convergence property:
		\begin{equation}\label{WeakConver}
e^{it^{j}_{n}\L_{a}}[(\lambda^{j}_{n})^{\frac{1}{2}}W^{J}_{n}(\lambda^{j}_{n}x+x^{j}_{n})]\rightharpoonup 0\quad \mbox{in}\,\,
\dot{H}^{1}_{a}, \quad \mbox{for all $1\leq j\leq J$.}
     \end{equation}
			\item Asymptotic orthogonality: for all $1\leq j\neq k\leq J^{\ast}$
	\begin{equation}\label{Ortho}
\lim_{n\rightarrow \infty}\left[ \frac{\lambda^{j}_{n}}{\lambda^{k}_{n}}+\frac{\lambda^{k}_{n}}{\lambda^{j}_{n}} 
+\frac{|x^{j}_{n}-x^{k}_{n}|^{2}}{\lambda^{j}_{n}\lambda^{k}_{n}}+
\frac{|t^{j}_{n}(\lambda^{j}_{n})^{2}-t^{k}_{n}(\lambda^{k}_{n})^{2}|}{\lambda^{j}_{n}\lambda^{k}_{n}}\right]=\infty.
\end{equation}		
		\item Asymptotic Pythagorean expansions:
		\begin{align}\label{MassEx}
		&\sup_{J}\lim_{n\to\infty}\big[M(f_{n})-\sum^{J}_{j=1}M(\phi^{j}_{n})-M(W^{J}_{n})\big]=0,\\\label{EnergyEx}
		&\sup_{J}\lim_{n\to\infty}\big[E_{a}(f_{n})-\sum^{J}_{j=1}E_{a}(\phi^{j}_{n})-E_{a}(W^{J}_{n})\big]=0.
		\end{align}
	\end{itemize}
\end{theorem}

The first step is the following refined Strichartz estimate (see \cite[Lemma 3.6]{KillipMiaVisanZhangZheng}).

\begin{lemma}[Refined Strichartz]\label{RefinedStrichartz}
Let $a>-\frac{1}{4}+\frac{1}{25}$. For $f\in \dot{H}^{1}_{a}(\R^{3})$ we have
\[
\| e^{-it\L_{a}}f\|_{L^{10}_{t,x}(\R\times\R^{3})}\lesssim
\|  f \|^{\frac{1}{5}}_{\dot{H}^{1}_{a}(\R^{3})}\sup_{N\in 2^{\Z}}
\| e^{-it\L_{a}}f_{N}\|^{\frac{4}{5}}_{L^{10}_{t,x}(\R\times\R^{3})}.
\]
\end{lemma}

Using this estimate and combining the arguments of \cite[Proposition 3.7]{KillipMiaVisanZhangZheng} 
and \cite[Proposition 7.2]{KillipOhPoVi2017}, we can extract single bubbles of concentration as follows:
\begin{proposition}[Inverse Strichartz inequality]\label{InverseSI}
Let $a>-\frac{1}{4}+\frac{1}{25}$. Let $\left\{f_{n}\right\}_{n\in \N}$ be a sequence such that
\[
\limsup_{n\to\infty}\|  f_{n} \|_{{H}^{1}_{a}}=A<\infty
\quad \text{and}\quad 
\liminf_{n\to\infty}\| e^{-it\L_{a}}f\|_{L^{10}_{t,x}(\R\times\R^{3})}=\epsilon>0.
\]
Then, after passing to a subsequence in $n$, there exist $\phi\in \dot{H}^{1}_{x}$,
\[
\left\{\lambda_{n}\right\}_{n\in \N}\subset (0,\infty), \quad\left\{t_{n}\right\}_{n\in \N}\subset \R,
\quad \left\{x_{n}\right\}_{n\in \N}\subset \R^{3}
\]
such that the following statements hold:
\begin{enumerate}[label=\rm{(\roman*)}]
	\item $\lambda_{n}\to \lambda_{\infty}\in [0, \infty)$, and if $\lambda_{\infty}>0$ then
	$\phi\in{H}^{1}_{x}$.
	\item Weak convergence property:
\begin{equation}\label{weakprofile}
\lambda^{\frac{1}{2}}_{n}( e^{-it\L_{a}}f_{n})(\lambda_{n}x+x_{n})\rightharpoonup \phi(x)
\quad \text{weakly in}\quad 
\begin{cases} 
H^{1}(\R^{3}), \quad \text{if $\lambda_{\infty}>0$} \\
\dot{H}^{1}(\R^{3}), \quad \text{if $\lambda_{\infty}=0$}.
\end{cases} 
\end{equation}
\item Decoupling of norms:
\begin{align}\label{ConverH1}
&	\lim_{n\to\infty}\left\{\|f_{n}\|^{2}_{\dot{H}^{1}_{a}}-\|f_{n}-\phi_{n}\|^{2}_{\dot{H}^{1}_{a}}\right\}
\gtrsim_{\epsilon,A} 1\\\label{ConverLp}
&\lim_{n\to\infty}\left\{\|f_{n}\|^{2}_{L_{x}^{2}}-\|f_{n}-\phi_{n}\|^{2}_{L_{x}^{2}}-\|\phi_{n}\|^{2}_{L_{x}^{2}}\right\}=0,
\end{align}
where {\small
\[
\phi_{n}(x):=\begin{cases} 
\lambda^{-\frac{1}{2}}_{n}e^{-it_{n}\L_{a}}\left[\phi\(\frac{x-x_{n}}{\lambda_{n}}\)\right] 
& \text{if $\lambda_{\infty}>0$}, \\
 \lambda^{-\frac{1}{2}}_{n}e^{-it_{n}\L_{a}}\left[(P^{a}_{\geq \lambda^{\theta}_{n}}\phi)\(\frac{x-x_{n}}{\lambda_{n}}\)\right]                 
& \text{if $\lambda_{\infty}=0$},
\end{cases} 
\]
with $0<\theta<1$.
}
\item We may choose the parameters $\left\{\lambda_{n}\right\}_{n\in \N}$, $\left\{t_{n}\right\}_{n\in \N}$  and $\left\{x_{n}\right\}_{n\in \N}$ such that either 
$\frac{t_{n}}{\lambda^{2}_{n}}\to \pm \infty$ or $t_{n}\equiv 0$ and either $\frac{|x_{n}|}{\lambda_{n}}\to  \infty$ 
or $x_{n}\equiv 0$.
\end{enumerate}
\end{proposition}
Arguing as in \cite[Corollary 7.3 (i)]{KillipOhPoVi2017}, \cite[Proposition 3.7]{KillipMiaVisanZhangZheng}, and \cite[Lemma 7.4]{KillipOhPoVi2017}, we also have the following: 
\begin{lemma}\label{ScalingPara}
Under the hypotheses of Proposition \ref{InverseSI}, we have:
\begin{enumerate}[label=\rm{(\roman*)}]
\item Passing to subsequence, we may assume that either $\lambda_{n}\equiv1$ or $\lambda_{n}\to0$.
\item \begin{align}\label{ConverL6}
&\lim_{n\to\infty}\left\{\|f_{n}\|^{6}_{L_{x}^{6}}-\|f_{n}-\phi_{n}\|^{6}_{L_{x}^{6}}-\|\phi_{n}\|^{6}_{L_{x}^{6}}\right\}=0,\\\label{ConverL4}
&\lim_{n\to\infty}\left\{\|f_{n}\|^{4}_{L_{x}^{4}}-\|f_{n}-\phi_{n}\|^{4}_{L_{x}^{4}}-\|\phi_{n}\|^{4}_{L_{x}^{4}}\right\}=0.
\end{align}
\end{enumerate}
\end{lemma}

With Proposition~\ref{InverseSI} and Lemma~\ref{ScalingPara} in place, the proof of Theorem~\ref{LinearProfi} then follows as in \cite[Theorem 7.5]{KillipOhPoVi2017}  (see also \cite[Theorem 3.1]{KillipMiaVisanZhangZheng} for a similar result in the energy-critical case).

%

\subsection{Embedding nonlinear profiles}\label{Sec:embedding}

In this section we construct scattering solutions to \eqref{NLS} associated to profiles $\phi_n$ living either at small length scales (i.e. in the regime $\lambda_n\to 0$) or far from the origin relative to their length scale (i.e. in the regime $|\tfrac{x_{n}}{\lambda_{n}}|\to \infty$), or both.  The challenge lies in the fact that the translation and scaling symmetries in \eqref{NLS} are broken by the potential and the double-power nonlinearity, respectively.  In particular, we must consider several limiting regimes and use approximation by a suitable underlying model in each case.  The basic idea is that if $\lambda_n\to 0$, the cubic term becomes negligible, while if $|\tfrac{x_n}{\lambda_n}|\to\infty$, the potential term becomes negligible. In particular: 
\begin{itemize}
\item If $\lambda_n\to 0$ and $x_n\equiv 0$, we approximate using solutions to the quintic NLS with inverse-square potential, that is, \eqref{CriticalNLS}.  For this model, scattering holds for arbitrary $\dot H^1$ data (cf. Theorem~\ref{CriticalWP}).
\item If $\lambda_n\equiv 1$ and $|x_n|\to\infty$, we approximate using solutions to the cubic-quintic NLS without potential, that is, \eqref{NLSfree}. For this model, scattering holds for data with mass-energy in the region $\mathcal{K}_0$ (see \cite{KillipOhPoVi2017}), which contains our desired scattering region $\mathcal{K}_a$ for all $a$ (cf. Corollary~\ref{Compa22}).
\item If $\lambda_n\to 0$ and $|\tfrac{x_n}{\lambda_n}|\to 0$, we approximate using solutions to the quintic NLS without potential, that is, \eqref{3dquintic}.  For this model, scattering holds for arbitrary $\dot H^1$ data \cite{KiiVisan2008, Bourgain1999, TaoKell2008}.
\end{itemize}
The technique of proof blends ideas from the works \cite{KillipOhPoVi2017, KillipMurphyVisanZheng2017, KillipMiaVisanZhangZheng}.

\begin{proposition}[Embedding nonlinear profiles]\label{P:embedding} Fix $a>-\tfrac14+\tfrac1{25}$.

Suppose $\lambda_n\equiv 1$ or $\lambda_n\to 0$, and that $\{x_n\}$ is such that either
\[
|\tfrac{x_n}{\lambda_n}|\to \infty, \qtq{or} \lambda_n\to 0 \qtq{and} x_n\equiv 0.
\]

Let $\L_a^n$  be as in \eqref{DefOperator} corresponding to sequence 
$\{\tfrac{x_{n}}{\lambda_{n}}\}$, and let $\{t_n\}$ satisfy $t_n\equiv 0$ or $t_n\to\pm\infty$.

\begin{itemize} 
\item If $\lambda_n\equiv 1$, then let $\phi\in H^1$ satisfy  {$(M(\phi),E_{0}(\phi))\in\mathcal{K}_0$} and define
\[
\phi_{n}(x):=[e^{-it_{n}\L^{n}_{a}}\phi](x-x_{n}).
\]
\item If $\lambda_n\to 0$, then let $\phi \in \dot H^1$, $\theta\in(0,1)$, and  
\[
\phi_{n}(x):=\lambda_{n}^{-\frac{1}{2}}[e^{-it_{n}\L^{n}_{a}} P^{a}_{\geq \lambda_{n}^{\theta}}\phi]( \tfrac{x-x_{n}}{\lambda_{n}}).
\]
\end{itemize}

Then for $n$ sufficiently large, there exists a global solution $v_n$ to \eqref{NLS} with
\[
v_n(0)=\phi_n \qtq{and} \|v_n\|_{L_{t,x}^{10}(\R\times\R^3)}\lesssim 1,
\]
with the implicit constant depending on $\|\phi\|_{H^1}$ if $\lambda_n\equiv 1$ or $\|\phi\|_{\dot H^1}$ if $\lambda_n\to 0$. 

Moreover, for any $\epsilon>0$ there exist $N=N(\epsilon)\in \N$ and a smooth compactly supported function
$\chi_{\epsilon}\in C^{\infty}_{c}(\R\times \R^{3})$  such that for  $n\geq N$,
\begin{align}\label{BewAprox11}
\Big\| v_{n}(t,x)-\lambda^{-1/2}_{n}\chi_{\epsilon}(\tfrac{t}{\lambda^{2}_{n}}+t_{n}, \tfrac{x-x_{n}}{\lambda_{n}})  \Big\|_{X(\R\times \R^{3})}
&<\epsilon,
\end{align}
 where 
\[
X\in\{L_{t,x}^{10}, L^{10}_{t}\dot{H}_{x}^{1,\frac{30}{13}}, L_t^{\frac{5}{2}}\dot H_x^{1,\frac{30}{13}} \}.
\]
\end{proposition} 

\begin{remark} In the scenario in which $\lambda_n\equiv 1$, the approximation in \eqref{BewAprox11} may also be taken to hold in Strichartz spaces of $L^2$ regularity. 
\end{remark}

\begin{proof} We distinguish three scenarios throughout the proof:
\begin{itemize}
\item Scenario Q$a$: $\lambda_n\to 0$ and $x_n\equiv 0$. (Here $\phi\in \dot H^1$.)
\item Scenario CQ0: $\lambda_n\equiv 1$ and $|x_n|\to\infty$. (Here $\phi\in H^1$.)
\item Scenario Q0: $\lambda_n\to 0$ and $|\tfrac{x_n}{\lambda_n}|\to\infty$ (Here $\phi\in \dot H^1$.) 
\end{itemize}
Fixing $\mu,\theta\in(0,1)$, we firstly define
\[
\psi_n = \begin{cases} P_{>\lambda_n^\theta}^a \phi &\text{in Scenario Q$a$,} \\ P_{\leq |x_n|^\mu}^a \phi & \text{in Scenario CQ0,} \\ P_{\lambda_n^\theta\leq \cdot<|\frac{x_n}{\lambda_n}|^\mu}^a \phi & \text{in Scenario Q0.}\end{cases}
\]
We also set
\[
H= \begin{cases} -\Delta & \text{in Scenarios CQ0 and Q0}, \\ \L_a & \text{in Scenario Q$a$}.
\end{cases} 
\]

\textbf{Construction of approximate solutions, part 1.} We first construct functions $w_n$ and $w$ as follows:

If $t_n\equiv 0$, then we define $w_n$ and $w$ as the global solutions to an appropriate NLS model with initial data $\psi_n$ and $\phi$, respectively.  In particular, in Scenario Q$a$, we use the model \eqref{CriticalNLS} (quintic NLS with inverse-square potential), appealing to Theorem~\ref{CriticalWP}.  In Scenario CQ0, we use the model \eqref{NLSfree} (cubic-quintic NLS without potential), appealing to the main result in \cite{KillipOhPoVi2017}.  Notice that $(M(\psi_n),E_{0}(\psi_n))\in\mathcal{K}_0$ for $n$ 
sufficiently large (recall that $|x_{n}|\to \infty$). Finally, in Scenario Q0, we use the model \eqref{3dquintic} (quintic NLS without potential), appealing to the main result in \cite{TaoKell2008}. 

If instead $t_n\to\pm\infty$, then we define $w_n$ and $w$ to be the solutions to the appropriate model (determined according to the three scenarios as above) satisfying
\begin{equation}\label{wnscat}
\|w_n-e^{-itH}\psi_n\|_{\dot H^1} \to 0 \qtq{and} \|w-e^{-itH}\phi\|_{\dot H^1} \to 0
\end{equation}
as $t\to\pm\infty$.  Note that in either case (i.e. $t_n\equiv 0$ or $t_n\to\pm\infty$), $w$ has scattering states $w_\pm$ as $t\to\pm\infty$ in $\dot H^1$. 

The solutions just constructed obey 
\begin{equation}\label{wnscat-bds}
\|\sqrt{H} w_n\|_{S^0(\R)}+\|\sqrt{H} w\|_{S^0(\R)} \leq C(\|\phi\|_{\dot H^1})
\end{equation}
uniformly in $n$.  At the level of $L^2$ regularity, by the Bernstein inequality and equivalence of Sobolev spaces (in Scenario Q$a$),  we may derive the following bounds:
\begin{equation}\label{JPRlo}
\begin{cases}
\|w_n\|_{S_a^0(\R)}\lesssim C(\|\phi\|_{\dot H^1})\lambda_n^{-\theta} & \text{in Scenario Q$a$}, \\
\|w_n\|_{S^0(\R)} \lesssim C(\|\phi\|_{H^1}) & \text{in Scenario CQ$0$}, \\
\|w_n\|_{S^0(\R)}\lesssim C(\|\phi\|_{\dot H^1}) \lambda_n^{-\theta} & \text{in Scenario Q$0$},
\end{cases}
\end{equation}
uniformly in $n$.  In Scenarios CQ$0$ and Q$0$, we may also use persistence of regularity to obtain the bounds
\begin{equation}\label{JPRhi}
\| |\nabla|^s w_n\|_{\dot S^1(\R)} \lesssim \bigl|\tfrac{x_n}{\lambda_n}\bigr|^{s\mu} 
\end{equation}
for higher $s$. 

By stability theory, we may also derive that in each case
\begin{equation}\label{wn-vs-w-stable}
\lim_{n\to\infty} \|\sqrt{H}[w_n-w]\|_{L_t^q L_x^r(\R\times\R^3)}=0\qtq{for all admissible}(q,r). 
\end{equation}

\textbf{Construction of approximate solutions, part 2.} We now define approximate solutions to \eqref{NLS} on $\R\times\R^3$:

For each $n$, let $\chi_n$ be a smooth function obeying
\[
\chi_n(x)=\begin{cases} 0 & |x_n+\lambda_n x|<\tfrac14|x_n| \\ 1 & |x_n + \lambda_n x|>\tfrac12 |x_n|,\end{cases}\qtq{with} |\partial^k \chi_n(x)|\lesssim \bigl(\tfrac{\lambda_n}{|x_n|}\bigr)^{|k|}
\]
uniformly in $x$.  In particular, $\chi_n(x)\to 1$ as $n\to\infty$ for each $x\in\R^3$.  In fact, in Scenario~Q$a$, we have $x_n\equiv 0$, so that $\chi_n(x)\equiv 1$ and the derivatives of $\chi_n$ vanish identically. 

Now, for $T\geq 1$, we define
\[
\tilde v_{n,T}(t,x)=\begin{cases} 
\lambda^{-\frac{1}{2}}_{n}[\chi_{n}w_{n}](\lambda^{-2}_{n}t, \lambda^{-1}_{n}(x-x_{n})), & |t|\leq \lambda^{2}_{n}T,\\
e^{-i(t-\lambda^{2}_{n}T)\L_{a}}\tilde{v}_{n,T}(\lambda^{2}_{n}T,x),& t> \lambda^{2}_{n}T,\\
e^{-i(t+\lambda^{2}_{n}T)\L_{a}}\tilde{v}_{n,T}(-\lambda^{2}_{n}T,x),& t<- \lambda^{2}_{n}T.
\end{cases} 
\]
In Scenario Q$a$, we alter the definition by using the first approximation for all $t\in\R$; in particular, the additional parameter $T$ plays no role in this scenario.  

Keeping in mind that $\tilde v_{n,T}$ are meant to be approximate solutions to \eqref{NLS}, we define the `errors'
\[
e_{n,T}:=(i\partial_{t}-\L_{a})\tilde{v}_{n,T}-|\tilde{v}_{n,T}|^{4}\tilde{v}_{n,T}+|\tilde{v}_{n,T}|^{2}\tilde{v}_{n,T}.
\]

\textbf{Conditions for stability.} Our goal is to establish the following: for $s\in \left\{1,\frac{3}{5}\right\}$,
\begin{align}\label{estimate11}
&\limsup_{T\to\infty}\limsup_{n\to\infty}\bigl\{\|\tilde{v}_{n,T}\|_{L^{\infty}_{t}H^{1}_{x}(\R\times\R^3)}
+\|\tilde{v}_{n,T}\|_{L^{10}_{t,x}(\R\times\R^3)}\bigr\}\lesssim 1,\\	\label{estimate22}
&	\limsup_{T\to\infty}\limsup_{n\to\infty}\|\tilde{v}_{n,T}(\lambda^{2}_{n}t_{n})-\phi_{n}\|_{\dot H^s}=0,
\\\label{estimate33}
&\limsup_{T\to\infty}\limsup_{n\to\infty}\||\nabla|^{s}
e_{n,T}\|_{N(\R)}=0,
\end{align}
where space-time norms are over $\R\times\R^{3}$.

\textbf{Proof of \eqref{estimate11} (space-time bounds).}  First, by definition of $\tilde v_{n,T}$, Strichartz, \eqref{JPRlo}
\begin{align*}
\|\tilde v_{n,T}\|_{L_t^\infty L_x^2} \lesssim \lambda_n \|\chi_n\|_{L_x^\infty} \|w_n\|_{L_t^\infty L_x^2} \lesssim \lambda_n^{1-\theta}.
\end{align*}
Similarly, noting that $\chi_n\equiv 1$ in Scenario Q$a$ and
\[
\|\nabla \chi_n\|_{L^\infty} \lesssim \tfrac{\lambda_n}{|x_n|}\to 0\qtq{as}n\to\infty
\]
in the remaining scenarios and using equivalence of Sobolev spaces, we may estimate 
\begin{align*}
\|\nabla& \tilde v_{n,T}\|_{L_t^{10}L_x^{\frac{30}{13}}\cap L_t^\infty L_x^2} \\
& \lesssim \| \nabla[\chi_n w_n]\|_{L_t^{10} L_x^{\frac{30}{13}}\cap L_t^\infty L_x^2} + \|[\chi_n w_n](\pm \lambda_n^2 T)\|_{\dot H^1} \\
&\lesssim  \|w_n\|_{L_t^{10}\dot{H}^{1,\frac{30}{13}}\cap L_t^\infty \dot{H}^{1}} + \|\nabla \chi_n\|_{L^3}\|w_n\|_{L_t^\infty L_x^6}+\|\chi_n\|_{L^\infty}\|\nabla w_n\|_{L_t^\infty L_x^2} \lesssim 1. 
\end{align*}
Thus, using Sobolev embedding as well, we derive \eqref{estimate11}.

\textbf{Proof of \eqref{estimate22} (agreement of data).}  We first observe that in all scenarios, we have the estimates
\[
\|\phi_n\|_{L^2} \lesssim 1 \qtq{and} \|\tilde v_{n,T}\|_{L_t^\infty L_x^2} \lesssim 1,
\]
so that it suffices to prove the $s=1$ case of \eqref{estimate22}. 

First, if $t_n\equiv 0$, then we first change variables to obtain
\[
\|\nabla[\tilde v_{n,T}(0)-\phi_n]\|_{L^2} = 
\begin{cases}
0 & \text{in Scenario Q$a$} \\
\bigl\|\nabla[\chi_nP^a_{\leq |x_n|^\mu}\phi - \phi]\bigr\|_{L^2} & \text{in Scenario CQ$0$} \\
\bigl  \|\nabla[\chi_n P^a_{\lambda_n^\theta\leq\cdot<|\tfrac{x_n}{\lambda_n}|^\mu}\phi-P_{>\lambda_n^\theta}^a\phi]\bigr\|_{L^2} & \text{in Scenario Q$0$}.
 \end{cases}
\]
We treat Scenario Q$0$ in detail and omit details for the simpler Scenario CQ$0$.  In particular, in Scenario Q$0$ we  
rewrite
\begin{align}
\chi_n  P_{\lambda_n^\theta\leq\cdot<|\frac{x_n}{\lambda_n}|^\mu}^a \phi - P^{a}_{>\lambda_n^\theta}\phi 
& = (\chi_n-1)P_{>\lambda_n^\theta}^a \phi \label{JDA1}\\
& \quad - \chi_n P_{>|\frac{x_n}{\lambda_n}|^\mu}^a\phi \label{JDA2}
\end{align}

For \eqref{JDA1}, we apply the product rule and write
\[
\nabla\eqref{JDA1} = \nabla\chi_n\cdot\phi+ (1-\chi_n)\nabla \phi  - \nabla\chi_n\cdot P_{\leq \lambda_n^\theta}^a\phi  - (1-\chi_n)\nabla P_{\leq \lambda_n^\theta}^a\phi. 
\]
For the first two terms, we have
\begin{align*}
\|\nabla \chi_n \phi + (1-\chi_n)\nabla\phi\|_{L^2} & \lesssim \|\nabla\chi_n\|_{L^3}\|\phi\|_{L^6(\text{supp}(\nabla \chi_n))} + \|\nabla \phi\|_{L^2(\text{supp}(1-\chi_n))} \\
& \lesssim \|\phi\|_{L^6(\text{supp}(\nabla \chi_n))} + \|\nabla \phi\|_{L^2(\text{supp}(1-\chi_n))} = o(1)
\end{align*}
as $n\to\infty$ by the dominated convergence theorem. For the last two terms, we instead have
\begin{align*}
\|\nabla&\chi_n\cdot P_{\leq \lambda_n^\theta}^a\phi  + (1-\chi_n)\nabla P_{\leq \lambda_n^\theta}^a\phi\|_{L^2} \\
& \lesssim \|\nabla \chi_n\|_{L^3}\|P_{\leq\lambda_n^\theta}^a\phi\|_{L^6} + \|\nabla P_{\leq \lambda_n^\theta}^a\phi\|_{L^2} \\ 
& \lesssim \|P_{\leq\lambda_n^\theta}^a\phi\|_{L^6} + \|\sqrt{\L_a}P_{\leq\lambda_n^\theta}^a\phi\|_{L^2} = o(1)
\end{align*}
as $n\to\infty$ by a density argument, using the fact that $\lambda_n\to 0$.  Applying the product rule to \eqref{JDA2} and then estimating as we just did for the last two terms shows that
\[
\|\nabla[\chi_n P_{>|\frac{x_n}{\lambda_n}|^\mu}^a \phi]\|_{L^2} \to 0 \qtq{as}n\to\infty,
\]
as well.  Thus, in the case $t_n\equiv 0$, we have
\[
\lim_{n\to\infty} \|\nabla[\tilde v_{n,T}(0)-\phi_n]\|_{L^2} =0.
\]

We next establish $\dot H^1$ convergence in the case $t_n\to+\infty$ (the case $t_n\to-\infty$ is handled similarly). As before, we change variables to obtain
\begin{align*}
\|\tilde v_{n,T}&(\lambda_n^2 t_n)-\phi_n\|_{\dot H_a^1} \\
&=\begin{cases} \|\sqrt{\L_a}\bigl[w_n(t_n)-e^{-it_n\L_a}P^{a}_{>\lambda_n^\theta}\phi\bigr]\|_{L^2} & \text{in Scenario Q$a$,} \\
\|\sqrt{\L_a^n}\bigl[(\chi_n w_n)(T)-e^{-iT\L_a^n}P_{> \lambda_n^\theta}^a\phi\bigr]\|_{L^2} & \text{in Scenario Q$0$}, \\
\|\sqrt{\L_a^n}\bigl[(\chi_n w_n)(T)-e^{-iT\L_a^n}\phi\bigr]\|_{L^2} & \text{in Scenario CQ$0$}.
\end{cases}
\end{align*}

In Scenario Q$a$, we have $P_{>\lambda_n^\theta}^a\phi = \psi_n$, and hence we obtain 
\[
\lim_{n\to\infty} \|\tilde v_{n,T}(\lambda_n^2 t_n)-\phi_n\|_{\dot H_a^1} = 0
\]
directly from \eqref{wnscat}. 

Again, let us treat Scenario Q$0$ in detail and omit details for the simpler Scenario CQ$0$.  We begin by using the equivalence of Sobolev spaces to obtain
\begin{align}
\|\tilde  v_{n,T}(\lambda_n^2 t_n)-\phi_n\|_{\dot H_a^1} 
& \lesssim \|\nabla[\chi_n(w_n(T)-w(T))]\|_{L^2} \label{JDAA1} \\
& \quad + \|\nabla[w(T)(\chi_n-1)]\|_{L^2} \label{JDAA2} \\
& \quad + \|\sqrt{\L_a^n}[w(T)-e^{-iT\L_a^n}\phi]\|_{L^2}\label{JDAA3} \\
& \quad + \| P_{\leq \lambda_n^\theta}^a\phi\|_{\dot H_a^1} \label{JDAA4}.
\end{align}

For \eqref{JDAA1}, we use H\"older's inequality and \eqref{wn-vs-w-stable} to obtain
\begin{align*}
\eqref{JDAA1} & \lesssim \|\nabla \chi_n\|_{L^3}\|w_n(T)-w(T)\|_{L^6} + \|\chi_n\|_{L^\infty}\|\nabla[w_n(T)-w(T)]\|_{L^2} \\
& \to 0 \qtq{as}n\to\infty.
\end{align*}
For \eqref{JDAA2}, we argue as above to obtain
\begin{align*}
\|\nabla[w(T)(\chi_n-1)]\|_{L^2} & \lesssim \|\nabla w(T)\|_{L^2(\text{supp}(\chi_n-1))} + \|w(T)\|_{L^6(\text{supp}(\nabla \chi_n))} \\
& \to 0 \qtq{as}n\to\infty.
\end{align*}
To estimate \eqref{JDAA3}, we decompose further and first write
\begin{align}
\eqref{JDAA3} & \lesssim \|(\sqrt{\L_a^n}-\sqrt{H})w(T)\|_{L^2} + \|[\sqrt{\L_a^n}-\sqrt{H}]\phi\|_{L^2} \label{JDAA31} \\
& \quad + \|(e^{-iT\L_a^n}-e^{-iTH})\sqrt{H}\phi\|_{L^2} \label{JDAA32} \\
& \quad + \|\sqrt{H}(w(T)-e^{-iTH}\phi)\|_{L^2}\label{JDAA33}.
\end{align}
We now observe that the terms in \eqref{JDAA31} tend to zero as $n\to\infty$ as a consequence \eqref{Conver33}.  The term in \eqref{JDAA32} tends to zero as $n\to\infty$ due to \eqref{Conver44}, while the term in \eqref{JDAA33} tends to zero as $T\to\infty$ due to \eqref{wnscat}.  Finally, a density argument and the fact that $\lambda_n\to 0$ imply that the term in  \eqref{JDAA4} tends to zero as as $T\to\infty$. 

This completes the proof of \eqref{estimate22}.

\textbf{Proof of \eqref{estimate33} (control of errors).} We consider each scenario separately.

\underline{\emph{Proof of \eqref{estimate33}} in Scenario Q$a$}. In Scenario Q$a$, we have 
\[
e_n = \lambda_n^{-\frac32} (|w_n|^2 w_n)(\lambda_n^{-2}t,\lambda_n^{-1}x),
\]
where we have dropped the subscript $T$, as it is irrelevant in this scenario.  By a change of variables, \eqref{JPRlo}, and \eqref{wnscat-bds}, we may now estimate
\begin{align*}
\|\nabla e_n\|_{L_t^{\frac53}L_x^{\frac{30}{23}}} & \lesssim \lambda_n \|w_n\|_{L_{t,x}^{10}}\|w_n\|_{L_t^{\frac52}L_x^{\frac{30}{7}}}\|\nabla w_n\|_{L_t^{10}L_x^{\frac{30}{13}}} \\
& \lesssim \lambda_n^{1-\theta} \to 0 \qtq{as}n\to\infty.
\end{align*}
Similarly, we derive 
\[
\|e_n\|_{L_t^{\frac{5}{3}}L_x^{\frac{30}{23}}} \lesssim \lambda_n^{2-2\theta}\to 0 \qtq{as}n\to\infty.
\]
Thus we obtain \eqref{estimate33} in Scenario Q$a$.

\underline{\emph{Proof of \eqref{estimate33} in Scenario CQ$0$.}}  As $\tilde v_{n,T}$ is defined piecewise in time, we will treat the regions $|t|\leq T$ and $|t|>T$ separately (recall that in Scenario CQ$0$, we have $\lambda_n\equiv 1$). 

Recalling that $w_n$ is a solution to \eqref{NLSfree}, we find that on the region $|t|\leq T$, we have
\begin{align}
e_{n,T}(t,x)
&=-[(\chi_{n}-\chi^{3}_{n})|w_{n}|^{2}w_{n}](t, x-x_{n})\label{New11}\\
&\quad+[(\chi_{n}-\chi^{5}_{n})|w_{n}|^{4}w_{n}](t, x-x_{n})\label{New22}\\
&\quad+2[\nabla \chi_{n}\cdot \nabla w_{n}](t, x-x_{n})
+[\Delta \chi_{n} w_{n}](t, x-x_{n})\label{New33}\\
&\quad-\tfrac{a}{|x|^{2}}[\chi_{n}w_{n}](t, x-x_{n}).\label{New55}
\end{align}

In the region $t>T$, say, we instead have
\begin{align}\label{JLTe1}
e_{n,T} = -|\tilde v_{n,T}|^4 \tilde v_{n,T} + |\tilde v_{n,T}|^2 \tilde v_{n,T}.
\end{align}

We begin by estimating \eqref{New11}--\eqref{New55} on $[-T,T]\times\R^3$. 

For \eqref{New11}, we apply a change of variables and H\"older's inequality to estimate
\begin{align*}
\|\nabla& \eqref{New11}\|_{L_{t}^{\frac{5}{3}}L^{\frac{30}{23}}_{x}}
\\ &\lesssim \|(\chi_{n}-\chi^{3}_{n})|w_{n}|^{2}\nabla w_{n}\|_{L_{t}^{\frac{5}{3}}L^{\frac{30}{23}}_{x}}+
\|\nabla \chi_{n}(1-3\chi^{2}_{n})w_{n}^{3}\|_{L_{t}^{\frac{5}{3}}L^{\frac{30}{23}}_{x}}\\
&\lesssim \bigl[ \|  \nabla w_{n} \|_{L_{t}^{10}L^{\frac{30}{13}}_{x}} \| w_{n} \|_{L_{t}^{\frac{5}{2}}L^{\frac{30}{7}}_{x}} 
+\|\nabla\chi_{n}\|_{L^{3}_{x}} \| w_{n} \|_{L_{t}^{\frac{5}{2}}L^{\frac{30}{7}}_{x}} \| w_{n} \|_{L^{{10}}_{t, x}} \bigr]\\
&\quad \times\bigl[ \| w_{n}-w \|_{L^{{10}}_{t, x}}+\|w \|_{L^{{10}}_{t, x}([-T,T]\times\{|x+x_{n}|\leq \frac{|x_{n}|}{4}\} )}\bigr]
\\ & \quad \to 0 \qtq{as}n\to\infty,
\end{align*}
where we have applied the dominated convergence theorem and \eqref{wn-vs-w-stable}.  On the other hand,
\[
\|\eqref{New11}\|_{L_{t}^{\frac{5}{3}}L^{\frac{30}{23}}_{x}}\lesssim
\|  w_{n} \|_{L_{t}^{10}L^{\frac{30}{13}}_{x}} 
\| w_{n} \|_{L_{t}^{\frac{5}{2}}L^{\frac{30}{7}}_{x}} 
\| w_{n} \|_{L^{{10}}_{t, x}}\lesssim 1,
\]
and thus we obtain the desired estimates on \eqref{New11} by interpolation.

The remaining terms, namely \eqref{New22}, \eqref{New33}, and \eqref{New55} may be handled exactly as in Step 4 in \cite[Theorem 4.1]{KillipMiaVisanZhangZheng} (setting $\lambda_n\equiv 1$; cf. the estimates of (4.9)--(4.12) therein). Thus, we will only mention the main ideas here.  For \eqref{New22}, the argument is similar to the one used to estimate \eqref{New11}.  For \eqref{New33}--\eqref{New55}, we estimate in $L_t^1 L_x^2$, obtaining the crude bound $T$ from the integral in time.  One relies on the decay of derivatives of $\chi_n$ and of the potential on the support of $\chi_n(\cdot-x_n)$ as $n\to\infty$; in particular, these terms are ultimately negligible due to the fact that $|x_n|\to\infty$.  The first term in \eqref{New33} also explains the need for the high-frequency cutoff in the definition of $w_n$, as additional derivatives may land on the term $\nabla w_n$. Thus, for example, uising \eqref{JPRhi}, we end up with the term
\[
T\|\nabla \chi_n\|_{L^\infty}\|\Delta w_n\|_{L_t^\infty L_x^2} \lesssim T|x_n|^{\mu-1}\to 0 \qtq{as}n\to\infty. 
\]

For the term \eqref{JLTe1}, the essential fact that we need is
\begin{equation}\label{LimitAx}
\limsup_{T\to\infty}\limsup_{n\to\infty} \|e^{-it\L_a^n}[\chi_n w_n(T)]\|_{L_{t,x}^{10}((0,\infty)\times\R^3)}\to 0.
\end{equation}
To see this, we may again argue as in \cite[Theorem~4.1, Step 4]{KillipMiaVisanZhangZheng}.  The idea is that (by estimating much as we did for \eqref{JDAA1}--\eqref{JDAA2} above), we may obtain
\[
\|e^{-it\L_a^n}[\chi_n w_n(T)]\|_{L_{t,x}^{10}((0,\infty)\times\R^3)} = \|e^{-it\L_a^n}w(T)\|_{L_{t,x}^{10}((0,\infty)\times\R^3)} + o(1)
\]
as $n\to\infty$.  But now, using the facts that $w(T)$ has a scattering state $w_+$ and $\L_a^n$ converges to $-\Delta$ (in the sense made precise below), the desired convergence can be derived from Strichartz estimates (for $e^{it\Delta}w_+$ on $(T,\infty)$) and the monotone convergence theorem.

With \eqref{LimitAx} in place, we can then estimate the terms in \eqref{JLTe1} as follows.

First (choosing $s\in\{1,\tfrac35\}$), 
\begin{align*}
\||\nabla|^{s}&  | \tilde{v}_{n,T}|^{2} \tilde{v}_{n,T}\|_{L_{t}^{\frac{5}{3}}L^{\frac{30}{23}}_{x}
(\left\{t>\lambda^{2}_{n}T\right\}\times \R^{3})}\\
&\lesssim 
\||\nabla|^{s}  \tilde{v}_{n,T}\|_{L_{t}^{10}L^{\frac{30}{13}}_{x}(\left\{t>\lambda^{2}_{n}T\right\}\times \R^{3})}
\|  \tilde{v}_{n,T} \|_{L^{{10}}_{t, x}(\left\{t>\lambda^{2}_{n}T\right\}\times \R^{3})}\\
&\quad \times \|  \tilde{v}_{n,T} \|_{L^{\frac{5}{2}}_{t}L^{\frac{30}{7}}_{x}(\left\{t>\lambda^{2}_{n}T\right\}\times \R^{3})}\\
&\lesssim
\|e^{it \L^{n}_{a}}[\chi_{n}w_{n}(T)]\|_{L^{10}_{t,x}(0, \infty)\times \R^{3}}\to 0
\end{align*}
as $n\to\infty$ and $T\to \infty$. Similarly, 
\[
\begin{split}
\||\nabla|^{s} | \tilde{v}_{n,T}|^{4} \tilde{v}_{n,T}\|_{L_{t}^{2}L^{\frac{6}{5}}_{x}
(\left\{t>\lambda^{2}_{n}T\right\}\times \R^{3})}
&\lesssim \|  \tilde{v}_{n,T} \|^{4}_{L^{{10}}_{t, x}(\left\{t>\lambda^{2}_{n}T\right\}\times \R^{3})}\\
&\lesssim \| e^{-it\L^{n}_{a}}(\chi_{n}\omega_{n}(T))  \|^{4}_{L^{{10}}_{t, x}((0, \infty)\times \R^{3})}\to 0,
\end{split}
\]
as $n\to\infty$ and $T\to \infty$.  

This completes the proof of \eqref{estimate33} in Scenario CQ$0$. 

\underline{\emph{Proof of \eqref{estimate33} in Scenario Q$0$.}} Again, we treat the regions $|t|\leq\lambda_n^2$ and $|t|>\lambda_n^2 T$ separately. 

Recalling that $w_n$ is a solution to \eqref{3dquintic}, we find that on the region $|t|\leq\lambda_n^2 T$, we have
\begin{align}
e_{n,T}(t,x)&=\lambda^{-\frac{3}{2}}_{n}[\chi^{3}_{n}|w_{n}|^{2}w_{n}](\lambda^{-2}_{n}t, \lambda^{-1}_{n}(x-x_{n})) 
\label{e11}
\\
&\quad+\lambda^{-\frac{5}{2}}_{n}[(\chi_{n}-\chi^{5}_{n})|w_{n}|^{4}w_{n}](\lambda^{-2}_{n}t, \lambda^{-1}_{n}(x-x_{n}))\label{e22}\\
&\quad+2\lambda^{-\frac{5}{2}}_{n}[\nabla \chi_{n}\cdot \nabla w_{n}](\lambda^{-2}_{n}t, \lambda^{-1}_{n}(x-x_{n}))\label{e33}
\\
&\quad+\lambda^{-\frac{5}{2}}_{n}[\Delta \chi_{n} w_{n}](\lambda^{-2}_{n}t, \lambda^{-1}_{n}(x-x_{n}))\label{e44}
\\
&\quad-\lambda^{-\frac{1}{2}}_{n}\tfrac{a}{|x|^{2}}[\chi_{n}w_{n}](\lambda^{-2}_{n}t, \lambda^{-1}_{n}(x-x_{n})).\label{e55}
\end{align}

In the region $t>\lambda_n^2 T$, we again have
\begin{equation}\label{e66}
e_{n,T} = -|\tilde v_{n,T}|^4 \tilde v_{n,T} + |\tilde v_{n,T}|^2 \tilde v_{n,T}.
\end{equation}

The estimates for \eqref{e22}--\eqref{e55} once again follow as in Step 4 in \cite[Theorem 4.1]{KillipMiaVisanZhangZheng}; in particular, they exploit essentially the fact that $\tfrac{|x_n|}{\lambda_n}\to\infty$ (and the fact that we work on a finite time interval for terms \eqref{e33}--\eqref{e55}).  Once again, the high-frequency cutoff in $w_n$ is used to handle the situation when additional derivatives land on $\nabla w_n$ in \eqref{e33}.  On the other hand, the low frequency cutoff in $w_n$ is needed to handle the remaining term \eqref{e11}, which we turn to now. 

Changing variables and applying H\"older's inequality, \eqref{JPRhi}, and \eqref{JPRlo}, we obtain
\begin{align*}
\|\nabla\eqref{e11}\|_{L_{t}^{\frac{5}{3}}L^{\frac{30}{23}}_{x}}
&\lesssim \lambda_{n}\|\chi_{n}\|_{L^{\infty}_{x}}
\|\nabla w_{n}\|_{L_{t}^{10}L^{\frac{30}{13}}_{x}}
\|  w_{n} \|_{L^{{10}}_{t, x}}\| w_{n}\|_{L^{\frac{5}{2}}_{t}L^{\frac{30}{7}}_{x}}\\
&\quad+\lambda_{n}\|\chi_{n}\|^{2}_{L^{\infty}_{x}}
\|\nabla \chi_{n}\|_{L^{3}_{x}}
\| w_{n}\|^{2}_{L_{t,x}^{10}}\| w_{n}\|_{L_{t}^{\frac{5}{2}}L^{\frac{30}{7}}_{x}}\\
& \lesssim\lambda^{1-\theta}_{n}\to 0 \quad \text{as $n\to \infty$.}
\end{align*}
Similarly, 
\[
\|\eqref{e11}\|_{L_{t}^{\frac{5}{3}}L^{\frac{30}{23}}_{x}}
\lesssim \lambda^{2-2\theta}_{n}\to 0 \quad \text{as $n\to \infty$,}
\]
and thus we obtain the desired estimates by interpolation.

For \eqref{e66}, we once again begin by observing \eqref{LimitAx}.  Then the estimate of the quintic term follows essentially as in Scenario~CQ$0$, while for the cubic term we use Strichartz, \eqref{JPRlo}, and estimate as follows:
\begin{align*}
&\|\nabla | \tilde{v}_{n,T}|^{2} \tilde{v}_{n,T}\|_{L_{t}^{\frac{5}{3}}L^{\frac{30}{23}}_{x}
(\left\{t>\lambda^{2}_{n}T\right\}\times \R^{3})}\\
&\lesssim 
\|\nabla \tilde{v}_{n,T}\|_{L_{t}^{10}L^{\frac{30}{13}}_{x}(\left\{t>\lambda^{2}_{n}T\right\}\times \R^{3})}
\|  \tilde{v}_{n,T} \|_{L^{{10}}_{t, x}(\left\{t>\lambda^{2}_{n}T\right\}\times \R^{3})}\|  \tilde{v}_{n,T} \|_{L^{\frac{5}{2}}_{t}L^{\frac{30}{7}}_{x}(\left\{t>\lambda^{2}_{n}T\right\}\times \R^{3})}\\
&\lesssim
\|  \tilde{v}_{n,T} \|_{L^{\frac{5}{2}}_{t}L^{\frac{30}{7}}_{x}(\left\{t>\lambda^{2}_{n}T\right\}\times \R^{3})}
\lesssim \lambda_{n}\|w_{n}\|_{{L}^{\infty}_{t}{L}^{2}_{x}}\lesssim \lambda^{1-\theta}_{n}\to 0 \quad \text{as $n\to\infty$.}
\end{align*}
Similarly,
\[
\|| \tilde{v}_{n,T}|^{2} \tilde{v}_{n,T}\|_{L_{t}^{\frac{5}{3}}L^{\frac{30}{23}}_{x}
(\left\{t>\lambda^{2}_{n}T\right\}\times \R^{3})} \lesssim \lambda_n^{2-2\theta}\to 0\quad \text{as $n\to\infty$,} 
\]
and hence by interpolation we obtain the desired bounds.

This completes the proof of \eqref{estimate33} in Scenario~Q$0$. 

\textbf{Construction of true solutions.} With \eqref{estimate11}--\eqref{estimate33} in place, we may apply the stability result (Lemma~\ref{StabilityNLS}) to deduce the existence of a global solution $v_n$ to \eqref{NLS} with $v_n(0)=\phi_n$,
\[
\|v_n\|_{L_{t,x}^{10}(\R\times\R^3)}\lesssim 1 \qtq{uniformly in}n,
\]
and
\[
\limsup_{T\to\infty}\limsup_{n\to\infty}\|v_n(t-\lambda_n^2 t_n)-\tilde v_{n,T}(t)\|_{\dot S_a^s(\R)} =0 \qtq{for}s\in\{1,\tfrac35\}.
\]

\textbf{Approximation by compactly supported functions.}  The final statement of the proposition, namely, the approximation in various energy-critical spaces by compactly supported functions of space-time, follows from a density argument as in \cite[Proposition~8.3]{KillipOhPoVi2017} and  \cite[Theorem~4.1]{KillipMiaVisanZhangZheng}, and relies primarily on the fact that we have obtained uniform space-time bounds in such spaces.  Thus, we omit the details and conclude the proof here. \end{proof}


\subsection{Existence of minimal blow-up solutions}\label{MinimalSolutions}
For each $\tau>0$,  we define
\[
B(\tau):=\sup\left\{\|  u \|_{L^{10}_{t,x}(\R\times\R^{3})}:\mbox{$u$ solves \eqref{NLS} and $\F(u)\leq \tau$}\right\},
\]
where $\F$ is as in \eqref{F-functional}.  By Lemma~\ref{FunctionF}(i), Theorem~\ref{MainTheorem} is equivalent to $B(\tau)<\infty$ for all $0<\tau<\infty$.

By Proposition~\ref{SDC} and \eqref{Enl}, we have $B(\tau)<\infty$ for $\tau$ sufficiently small.  Thus, by monotonicity of $B$, there exists $0<\tau_c\leq\infty$ so that
\begin{equation}\label{CriticalLevel}
\tau_{c}=\sup\left\{\tau: B(\tau)<\infty\right\}=\inf\left\{\tau: B(\tau)=\infty\right\}.
\end{equation}
We assume towards a contradiction that $\tau_{c}< \infty$. Using Lemma~\ref{StabilityNLS}, this implies that $B(\tau_c)=\infty$. Thus there exists a sequence of solutions $u_{n}$ such that  $\F(u_{n})\rightarrow \tau_{c}$ and $\|  u_{n} \|_{L^{10}_{t,x}(\R\times\R^{3})}\rightarrow \infty$ as $n\rightarrow\infty$. We will prove the existence of a solution $u_{c}\in H^{1}(\R^{3})$ such that $\F(u_{c})=\tau_{c}$, 
\begin{equation}\label{blowSolution}
\|  u_{c} \|_{L^{10}_{t,x}([0,\infty)\times\R^{3})}=\|  u_{c} \|_{L^{10}_{t,x}((-\infty,0]\times\R^{3})}=\infty,
\end{equation}
and such that
\begin{equation}\label{orbit-is-precompact}
\left\{u_c(t):t\in \R\right\}\qtq{is precompact in }H^{1}(\R^{3}).
\end{equation}

\begin{theorem}[Existence of minimal blow-up solutions]\label{CompacSolution} Suppose Theorem~\ref{MainTheorem} fails.  Then there exists a $u_{c,0}\in H^{1}(\R^{3})$ with $\F(u_{c,0})=\tau_{c}$
such that if $u_{c}$ is the corresponding solution to \eqref{NLS} with data $u_{c}(0)=u_{c,0}$, then \eqref{blowSolution} and \eqref{orbit-is-precompact} hold.
\end{theorem}

Arguing as in \cite[Theorem~9.6]{KillipOhPoVi2017}, to establish Theorem~\ref{CompacSolution}, it will suffice to establish the following Palais--Smale condition. 

\begin{proposition}\label{PScondition}
Let $\left\{u_{n}\right\}_{n\in \N}\subset H^{1}(\R^{3})$ be a sequence of solutions to \eqref{NLS} such that
$\lim_{n\rightarrow\infty}\F(u_{n})=\tau_{c}$, and suppose $t_{n}\in \R$ satisfy
\begin{equation}\label{Blowp2}
\lim_{n\rightarrow\infty}\|  u_{n} \|_{L^{10}_{t,x}([t_{n},\infty)\times\R^{3})}=\|  u_{n} \|_{L^{10}_{t,x}((-\infty,t_{n}]\times\R^{3})}=\infty.
\end{equation}
 Then we have that $\left\{u_{n}\right\}_{n\in \N}$ converges along a subsequence in  $H_{x}^{1}(\R^{3})$.
\end{proposition}
\begin{proof}
By time-translation invariance, we may assume that $t_{n}\equiv 0$. Using \eqref{Enl} and writing $u_{n,0}=u_n(0)$, we have 
\[
\|u_{n,0}\|^{2}_{H^{1}}  \lesssim \F(u_{n})  \lesssim \tau_{c}.
 \]
Applying Theorem~\ref{LinearProfi}, we may write
\begin{equation}\label{Dpe}
u_{n}(0)=\sum^{J}_{j=1}\phi^{j}_{n}+W^{J}_{n}
\end{equation}
for each $J\leq J^{\ast}$, with the various sequences satisfying \eqref{Reminder}--\eqref{EnergyEx}.  We may further assume that  $M(u_{n})\rightarrow M_{0}$, $E_{a}(u_{n})\rightarrow E_{0,a}$ and therefore $\tau_{c}=\F(M_{0}, E_{0,a})$ (cf. Lemma~\ref{FunctionF}).
By \eqref{MassEx} and \eqref{EnergyEx}, we also have 
\begin{align}\label{Pv11}
&\limsup_{n\rightarrow\infty}\sum^{J}_{j=1}M(\phi_{n}^{j})+M(W^{J}_{n})\leq M_{0},\\\label{Pv22}
&\limsup_{n\rightarrow\infty}\sum^{J}_{j=1}E_{a}(\phi_{n}^{j})+E_{a}(W^{J}_{n})\leq E_{0,a},
\end{align}
for each finite $J\leq J^{\ast}$, with all energies in \eqref{Pv22} nonnegative. Moreover, by \eqref{Enl} and the nontriviality of $\phi_{n}^{j}$ we have that $\liminf_{n\to\infty}E_{a}(\phi_{n}^{j})>0$.  

Our goal is to show that there can be at most nonzero $\phi_n^j$.

\textbf{Scenario 1.}  
\begin{equation}\label{Scena1}
\sup_{j}\limsup_{n\rightarrow\infty}M(\phi_{n}^{j})=M_{0}\quad
\mbox {and} \quad \sup_{j}\limsup_{n\rightarrow\infty}E_{a}(\phi_{n}^{j})=E_{0,a}.
\end{equation}
By \eqref{Pv22}, positivity of energy yields $J^{\ast}=1$.  In this case, we have that $W^{1}_{n}\rightarrow 0$ in $H_{x}^{1}$ as $n\rightarrow\infty$. Indeed, since $M(W^{1}_{n})\geq0$ and  $E_{a}(W^{1}_{n})\geq0$, we get  
$\limsup_{n\rightarrow\infty}E_{a}(W_{n}^{1})=0$ and $\limsup_{n\rightarrow\infty}M(W_{n}^{1})=0$. Thus,  \eqref{Enl} implies that $\limsup_{n\rightarrow\infty}\|W_{n}^{1}\|^{2}_{H_{x}^{1}}=0$.  In particular, we obtain
\begin{equation}\label{Csc}
u_{n}(0)=\phi^{1}_{n}+W_{n}^{1}, \quad\text{with}\quad \lim_{n\rightarrow\infty}\|W_{n}^{1}\|^{2}_{H_{x}^{1}}=0.
\end{equation}

Now suppose that $\frac{|x^{1}_{n}|}{\lambda^{1}_{n}}\to \infty$. Then Proposition~\ref{P:embedding} yields a global solution $v_{n}$ with $v_{n}(0)=\phi^{1}_{n}$ such that
\begin{equation*}
\|v_{n}\|_{L^{10}_{t,x}(\R\times\R^{3})}\lesssim 1.
\end{equation*}
As $W_{n}^{1}=u_{n}(0)-v_{n}(0)$, it follows that $\lim_{n\to\infty}\|u_{n}(0)-v_{n}(0)\|_{H_{x}^{1}}=0$. Thus, Lemma~\ref{StabilityNLS} implies that for $n$ large $u_{n}$ is a global solution with finite scattering norm, contradicting \eqref{Blowp2}. It follows that $x^{1}_{n}\equiv 0$.

Next, suppose that $\lambda^{1}_{n}\to 0$ as $n\to \infty$.  In this case, Proposition~\ref{P:embedding} yields a global solution $v_{n}$
with  $v_{n}(0)=\phi^{1}_{n}$ and $\|v_{n}\|_{L^{10}_{t,x}(\R\times\R^{3})}\lesssim 1$. Then Lemma~\ref{StabilityNLS} implies that $\|u_{n}\|_{L^{10}_{t,x}(\R\times\R^{3})}\lesssim 1$ for $n$ large enough, again contradicting \eqref{Blowp2}. It follows that $\lambda_n^1\equiv 1$.

Finally, suppose that $t^{1}_{n}\to \infty$ as $n\to \infty$.  By Sobolev embedding, Strichartz estimates, monotone convergence, and \eqref{Csc}, we deduce that
\begin{equation}\label{Eq11}
\begin{split}
\| & e^{-it\L_{a}}u_{n}(0)  \|_{L^{10}_{t,x}([0,\infty)\times \R^{3})} \\
&\leq \|e^{-it\L_{a}} \phi^{1}_{n}  \|_{L^{10}_{t,x}([0,\infty)\times \R^{3})}+
\|e^{-it\L_{a}} W_{n}^{1} \|_{L^{10}_{t,x}([0,\infty)\times \R^{3})},\\
&\lesssim
\|e^{-it\L_{a}} \phi^{1} \|_{L^{10}_{t,x}([t^{1}_{n},\infty)\times \R^{3})}+\|W_{n}^{1}\|_{H_{x}^{1}}\rightarrow 0,
\end{split}
\end{equation}
as $n\rightarrow\infty$. Writing $\tilde{u}_{n}=e^{-it\L_{a}}u_{n}(0)$ and $e_{n}=|\tilde{u}_{n}|^{4}\tilde{u}_{n}-|\tilde{u}_{n}|^{2}\tilde{u}_{n}$, we use \eqref{Eq11}, H\"older, and 
Strichartz to obtain
\[
\|\nabla e_{n}  \|_{N(\R)}\to 0\quad \text{as $n\to \infty$}.
\]
Thus Lemma~\ref{StabilityNLS} again leads to a contradiction with \eqref{Blowp2}. An analogous argument handles the case $t^{1}_{n}\rightarrow -\infty$ as $n\rightarrow\infty$.

Thus, in Scenario 1, we obtain that $x^{1}_{n}\equiv 0$, $t^{1}_{n}\equiv 0$ and $\lambda^{1}_{n}\equiv 1$.  This yields the desired conclusion of Proposition~\ref{PScondition}, and hence it remains to show that the only remaining scenario results in a contradiction. 

\textbf{Scenario 2.} If \eqref{Scena1} fails for all $j$, then there exists $\delta>0$ such that
\begin{equation}\label{Scena2}
\sup_{j}\limsup_{n\rightarrow\infty}M(\phi_{n}^{j})\leq M_{0}-\delta\quad
\mbox {or} \quad \sup_{j}\limsup_{n\rightarrow\infty}E_{a}(\phi_{n}^{j})\leq E_{0,a}-\delta.
\end{equation}

We then define nonlinear profiles $\psi_{n}^{j}$ associated to each $\phi_{n}^{j}$ as follows:
 \begin{itemize}
	\item If $\frac{|x^{j}_{n}|}{\lambda^{j}_{n}}\to \infty$ for some $j$, then we are in position to apply Proposition~\ref{P:embedding}, and hence we have a global solution $\psi_{n}^{j}$ of \eqref{NLS} with data $\psi_{n}^{j}(0)=\phi_{n}^{j}$. Indeed,  it is enough to show that $(M(\phi^{j}), E_{0}(\phi^{j}))\in \K_{0}$ when $\lambda_{n}^{j}\equiv 1$, 
	$|x_{n}^{j}| \to \infty$ and $t_{n}^{j}\equiv 0$. To see this, first note that by \eqref{MassEx}, \eqref{EnergyEx} and Lemma~\ref{FunctionF}(iii) we have
	$\F(\phi_{n}^{j})\leq \F(M_{0}, E_{0,a})= \tau_{c}$ for $n$ large. Thus, as $E_{a}(\phi_{n}^{j})\geq 0$ and $M(\phi_{n}^{j})=M(\phi^{j})$ we deduce that there exists 
	$\epsilon=\epsilon(j)>0$ such that 	\[
	\text{dist}\((M(\phi_{n}^{j}),E_{a}(\phi_{n}^{j})), \Omega_{a}\)\geq \epsilon.
	\]
(cf. \eqref{F-functional}.  Notice also that \eqref{Conver33} implies $\lim_{n\to\infty}E_{a}(\phi_{n}^{j})=E_{0}(\phi^{j})$ (recall that $|x_{n}^{j}| \to \infty$ and $t_{n}^{j}\equiv 0$).  Thus, the inequality above yields
\[
\text{dist}\((M(\phi^{j}),E_{0}(\phi^{j})), \Omega_{a}\)\geq \epsilon.
\]
In particular, as $\K_{a}\subseteq \K_{0}$, we deduce that $\text{dist}\((M(\phi^{j}),E_{0}(\phi^{j})), \Omega_{0}\)\geq \epsilon$, and hence
\[
E_{0}(\phi^{j})+\frac{M(\phi^{j})+E_{0}(\phi^{j})}{\mbox{dist}\((M(\phi^{j}),E_{0}(\phi^{j})),\Omega_{0}\)}<\infty.
\]
In view of Lemma~\ref{FunctionF}(i), this implies $(M(\phi^{j}), E_{0}(\phi^{j}))\in \K_{0}$.

	\item If $x^{j}_{n}\equiv 0$ and $\lambda^{j}_{n}\rightarrow 0$, we define $\psi_{n}^{j}$ to be the global solution of \eqref{NLS} with the initial data $\psi_{n}^{j}(0)=\phi_{n}^{j}$ guaranteed by Proposition~\ref{P:embedding}.
	
	\item If $x^{j}_{n}\equiv 0$, $\lambda^{j}_{n}\equiv 1$ and $t^{n}_{j}\equiv 0$, we take  $\psi^{j}$ to be the global solution of 
	\eqref{NLS} with the initial data $\psi^{j}(0)=\phi^{j}$.
	
\item If $x^{j}_{n}\equiv 0$, $\lambda^{j}_{n}\equiv 1$ and $t^{n}_{j}\rightarrow\pm\infty$, we take $\psi^{j}$  to be the global solution  of \eqref{NLS} that scatters to $e^{-it\L_{a}}\phi^{j}$ in $H^{1}_{x}(\R^{3})$  as $t\rightarrow\pm\infty$.  In either case, we define the global solution to \eqref{NLS},
\[\psi^{j}_{n}(t,x):=\psi^{j}(t+t^{j}_{n}, x).\]
\end{itemize}

By construction, we have that for each $j$, 
	\begin{equation}\label{Aproxi11}
\|\psi_{n}^{j}(0)- \phi_{n}^{j}\|_{H^{1}_{a}}\rightarrow0,\quad \text{as $n\rightarrow\infty$}.
\end{equation}
Moreover, notice that by \eqref{Scena2} and Lemma~\ref{FunctionF}(v), we may obtain 
\begin{equation}\label{BoundProfile}
\| \psi^{j}_{n}  \|_{L^{10}_{t,x}}\lesssim_{\delta, \tau_{c}} 1, \quad \text{for $n$ large and $1\leq j\leq J$}.
\end{equation}

In particular, by \eqref{BoundProfile}, \eqref{Enl} and 
Remark \ref{PRegularity}, we have the following: 
\begin{align}\label{ImporBound}
&\| \psi^{j}_{n}  \|_{L^{10}_{t,x}(\R\times\R^{3})}\lesssim_{\delta, \tau_{c}} [E_{a}(\psi^{j}_{n})]^{\frac{1}{2}},
\quad
\|\psi^{j}_{n}\|_{L^{10}_{t}\dot{H}_{a}^{1,\frac{30}{13}}(\R\times\R^{3})}\lesssim_{\delta, \tau_{c}} [E_{a}(\psi^{j}_{n})]^{\frac{1}{2}},\\
\label{ImporBound22}
&\|\psi^{j}_{n}\|_{L^{\frac{5}{2}}_{t}L_{x}^{\frac{30}{7}}(\R\times\R^{3})}\lesssim_{\delta, \tau_{c}} [M(\psi^{j}_{n})]^{\frac{1}{2}}.
\end{align}

We define the approximate solutions
\[
u^{J}_{n}(t):=\sum^{J}_{j=1}\psi^{j}_{n}(t)+e^{-it\L_{a}}W^{J}_{n},
\] 
with the goal of applying Lemma~\ref{StabilityNLS} to contradict \eqref{Blowp2}. In particular, we define the errors $e_n^J$ via 
\[
 (i\partial_{t}-\L_{a}) {u}^{J}_{n}=-|{u}^{J}_{n}|^{2}{u}^{J}_{n}+|{u}^{J}_{n}|^{4}{u}^{J}_{n}+e^{J}_{n}.
  \]

From \eqref{Aproxi11} we see that
\begin{equation}\label{DefiuJ}
\begin{split}
\lim_{n\to \infty}\|u_{n}^{J}(0)- u_{n}(0)\|_{H^{1}_{x}}=0, \quad \text{for any $J$}.
\end{split}
\end{equation}
It will suffice to establish the following estimates:
\begin{align}\label{Bound11}
&\sup_{J}\limsup_{n\rightarrow\infty}\| {u}^{J}_{n}  \|_{L_{t}^{\infty}H^{1}_{x}(\R\times \R^{3})}\lesssim_{\tau_{c},\delta} 1,
\\
\label{Bound22}
&\sup_{J}\limsup_{n\rightarrow\infty}\big[ \| {u}^{J}_{n}  \|_{L^{10}_{t,x}}
+ \|u_{n}^{J}\|_{L^{10}_{t}\dot{H}_{a}^{1,\frac{30}{13}}}+\|u_{n}^{J}\|_{L^{\frac{5}{2}}_{t}L_{x}^{\frac{30}{7}}}\big]
 \lesssim_{\tau_{c},\delta} 1,
\\
\label{Bound33}
&\lim_{J\to J^{\ast}}\limsup_{n\rightarrow\infty}\|\nabla e^{J}_{n}  \|_{N(\R)}=0,
\end{align}
where here and below all space-time norms are taken over $\R\times\R^3$. Indeed, using  \eqref{DefiuJ}, \eqref{Bound11}, \eqref{Bound22}, and \eqref{Bound33}, Lemma~\ref{StabilityNLS} implies that  $\| u_{n}  \|_{L^{10}_{t,x}}\lesssim_{\tau_{c},\epsilon,\delta}1$ for  $n$ large, contradicting \eqref{Blowp2}. 

We therefore turn to establishing the estimates \eqref{Bound11}-\eqref{Bound33}. We will use the following lemma.  

\begin{lemma}[Asymptotic decoupling]\label{AsymptoticDec}
If $j\neq k$ we have
\[
\lim_{n\to \infty}[
\|\psi^{j}_{n} \psi^{k}_{n}\|_{L^{5}_{t, x}}
+\| \psi^{j}_{n} \nabla\psi^{k}_{n} \|_{L^{5}_{t}{L}_{x}^{\frac{15}{8}}}
+\|\nabla \psi^{j}_{n} \nabla\psi^{k}_{n} \|_{L^{5}_{t}{L}_{x}^{\frac{15}{13}}}
+\| \psi^{j}_{n}\psi^{k}_{n} \|_{L^{\frac{5}{4}}_{t}{L}_{x}^{\frac{15}{7}}}
]=0.
\]
\end{lemma}

\begin{proof} As the proof follows essentially as in \cite[Lemma~4.1]{YANG2020124006} or \cite[Lemma~9.2]{KillipOhPoVi2017}, we will provide only a brief sketch here.  For first three terms, which involve energy-critical type spaces, the basic idea is to approximate the solutions by compactly supported functions of space-time (this requires uniform space-time bounds and relies on Proposition~\ref{P:embedding} when necessary), and then to exploit the orthogonality of parameters. For the fourth term, if $\lambda_n^j\equiv1$ and $\lambda_n^k\equiv 1$, then the solutions arise from $H^1$ profiles and we obtain space-time bounds (and hence the approximation result) in Strichartz spaces at $L^2$-regularity. If $\lambda_n^j\to 0$, then the solution arises from a frequency-truncated $\dot H^1$ profile and (again by persistence of regularity) has asymptotically vanishing space-time norms at $L^2$-regularity.  In particular, if one or both of the scales tends to zero, we obtain asymptotic vanishing by H\"older's inequality. \end{proof}

As \eqref{Bound11} readily follows from Strichartz \eqref{DefiuJ}, \eqref{Bound22}, and \eqref{Bound33}, it will suffice to establish \eqref{Bound22} and \eqref{Bound33}. 

\begin{proof}[Proof of \eqref{Bound22}] Let us show the estimate of the $L_{t,x}^{10}$-norm only, as the remaining terms may be handled in a similar fashion.  By \eqref{ImporBound}, equivalence of Sobolev spaces and Strichartz, we have get 
\[\begin{split}
\| u^{J}_{n}  \|^{2}_{L^{10}_{t,x}}&\lesssim \sum^{J}_{j=1}\| \psi^{j}_{n}  \|^{2}_{L^{10}_{x}}+
\sum_{j\neq k}\| \psi^{j}_{n} \psi^{k}_{n}  \|_{L^{5}_{t,x}}+\| e^{-it\L_{a}}W^{J}_{n}  \|^{2}_{L^{10}_{t,x}}\\
&\lesssim_{\tau_{c}}\sum^{J}_{j=1}E_{a}( \psi^{j}_{n} )+\sum^{J}_{j\neq k}o(1)+E_{a}(W^{J}_{n})
\lesssim_{\tau_{c}}1 +o(1)\cdot J^{2}
\end{split}
\]
as $n\to\infty$. 
\end{proof}

\begin{proof}[Proof of \eqref{Bound33}] Since $\psi^{j}_{n}$ is solution of \eqref{NLS} we can write
\begin{align}\label{Fdefi11}
e^{J}_{n}&= \sum^{J}_{j=1}F( \psi^{j}_{n} )-F\big(\sum^{J}_{j=1}\psi^{j}_{n}\big)\\\label{Fdefi22}
         &+F(u^{J}_{n}-e^{-it\L_{a}}W^{J}_{n})-F(u^{J}_{n}),
\end{align}
where $F(z)=F_{1}(z)-F_{2}(z)$ with $F_{1}(z):=|z|^{4}z$ and $F_{2}(z):=|z|^{2}z$. Now, by H\"older's inequality we have
\begin{align}\label{F1E}
\|\nabla\big[\sum^{J}_{j=1}F_{1}(\psi^{j}_{n}) -F_{1}(\sum^{J}_{j=1}\psi^{j}_{n})\big]\|_{L^{2}_{t}L^{\frac{6}{5}}_{x}}
&\lesssim \sum_{j\neq k}\| \psi^{j}_{n}  \|^{3}_{L^{10}_{x}}\|  \psi^{j}_{n} \nabla\psi^{k}_{n}  \|_{L^{5}_{t}L^{\frac{15}{8}}_{x}},\\
\label{F2E}
\|\nabla\big[\sum^{J}_{j=1}F_{2}(\psi^{j}_{n}) -F_{2}(\sum^{J}_{j=1}\psi^{j}_{n})\big]\|_{L^{\frac{5}{3}}_{t}L^{\frac{30}{23}}_{x}}
&\lesssim \sum_{j\neq k}\| \psi^{j}_{n}  \|_{L^{\frac{5}{2}}_{t}L^{\frac{30}{7}}_{x}}\|  \psi^{j}_{n} \nabla\psi^{k}_{n}  \|_{L^{5}_{t}L^{\frac{15}{8}}_{x}}.
\end{align}
Thus, by orthogonality, \eqref{ImporBound}, \eqref{ImporBound22}, \eqref{F1E} and  \eqref{F2E}  we get
\begin{equation}\label{FirstF}
\lim_{J\to J^{\ast}}\limsup_{n\rightarrow\infty}\| \nabla \eqref{Fdefi11}\|_{N(\R)}=0.
\end{equation}

We next estimate \eqref{Fdefi22}. First, by interpolation we get
\begin{align*}
&\|\nabla [F_{1}(u^{J}_{n}-e^{-it\L_{a}}W^{J}_{n})-F_{1}(u^{J}_{n})]\|_{L^{2}_{t}L^{\frac{6}{5}}_{x}}\\
&\lesssim 
	\|  e^{-it\L_{a}}W^{J}_{n} \|^{4}_{L^{10}_{t, x}}
\| \nabla  e^{-it\L_{a}}W^{J}_{n}  \|_{L^{10}_{t}{L}_{x}^{\frac{30}{13}}}
+	\|  e^{-it\L_{a}}W^{J}_{n} \|^{4}_{L^{10}_{t, x}}
\| \nabla u^{J}_{n}  \|_{L^{10}_{t}{L}_{x}^{\frac{30}{13}}}\\
&+	\|  e^{-it\L_{a}}W^{J}_{n} \|_{L^{10}_{t, x}}\| u^{J}_{n} \|^{3}_{L^{10}_{t, x}}
\| \nabla u^{J}_{n}  \|_{L^{10}_{t}{L}_{x}^{\frac{30}{13}}}+
\| u^{J}_{n} \|^{3}_{L^{10}_{t, x}}\|u^{J}_{n} \nabla  e^{-it\L_{a}}W^{J}_{n}  \|_{L^{5}_{t}{L}_{x}^{\frac{15}{8}}}
\end{align*}
Combining \eqref{Reminder}, \eqref{Pv22} and \eqref{ImporBound} we see that
\[\begin{split}
\lim_{J\to J^{\ast}}\limsup_{n\rightarrow\infty}\|\nabla [F_{1}(u^{J}_{n}-e^{-it\L_{a}}W^{J}_{n})-F_{1}(u^{J}_{n})]\|_{N^{1}(\R)}
\\
\lesssim 
\lim_{J\to J^{\ast}}\limsup_{n\rightarrow\infty}\|u^{J}_{n} \nabla  e^{-it\L_{a}}W^{J}_{n}  \|_{L^{5}_{t}{L}_{x}^{\frac{15}{8}}}.
\end{split}
\]
Similarly, 
\begin{align*}
&\|\nabla [F_{2}(u^{J}_{n}-e^{-it\L_{a}}W^{J}_{n})-F_{2}(u^{J}_{n})]\|_{L^{\frac{5}{3}}_{t}L^{\frac{30}{23}}_{x}}\\
&\lesssim 
\|  e^{-it\L_{a}}W^{J}_{n} \|_{L^{10}_{t, x}}
\| e^{-it\L_{a}}W^{J}_{n}  \|_{L^{\frac{5}{2}}_{t}{L}_{x}^{\frac{30}{7}}}
\| \nabla  e^{-it\L_{a}}W^{J}_{n}  \|_{L^{10}_{t}{L}_{x}^{\frac{30}{13}}}
\\
&\quad +	
\|  e^{-it\L_{a}}W^{J}_{n} \|_{L^{10}_{t, x}}
\| e^{-it\L_{a}}W^{J}_{n}  \|_{L^{\frac{5}{2}}_{t}{L}_{x}^{\frac{30}{7}}}
\| \nabla  u^{J}_{n}  \|_{L^{10}_{t}{L}_{x}^{\frac{30}{13}}}
\\
&\quad+
\|  e^{-it\L_{a}}W^{J}_{n} \|_{L^{10}_{t, x}}
\| u^{J}_{n}  \|_{L^{\frac{5}{2}}_{t}{L}_{x}^{\frac{30}{7}}}
\| \nabla  u^{J}_{n}  \|_{L^{10}_{t}{L}_{x}^{\frac{30}{13}}}
\\
& \quad +\| u^{J}_{n}  \|_{L^{\frac{5}{2}}_{t}{L}_{x}^{\frac{30}{7}}}
\|u^{J}_{n} \nabla  e^{-it\L_{a}}W^{J}_{n}  \|_{L^{5}_{t}{L}_{x}^{\frac{15}{8}}}.
\end{align*}
As Strichartz together with \eqref{Reminder}, \eqref{Pv22} and \eqref{ImporBound} implies
\[\begin{split}
\lim_{J\to J^{\ast}}\limsup_{n\rightarrow\infty}\|\nabla [F_{2}(u^{J}_{n}-e^{-it\L_{a}}W^{J}_{n})-F_{2}(u^{J}_{n})]\|_{N^{1}(\R)}
\\
\lesssim 
\lim_{J\to J^{\ast}}\limsup_{n\rightarrow\infty}\|u^{J}_{n} \nabla  e^{-it\L_{a}}W^{J}_{n}  \|_{L^{5}_{t}{L}_{x}^{\frac{15}{8}}},
\end{split}
\]
it remains to show
\begin{equation}\label{FinalLimit}
\lim_{J\to J^{\ast}}\limsup_{n\rightarrow\infty}\|u^{J}_{n} \nabla  e^{-it\L_{a}}W^{J}_{n}  \|_{L^{5}_{t}{L}_{x}^{\frac{15}{8}}}
=0.
\end{equation}

Applying H\"older we deduce
\[\begin{split}
\|u^{J}_{n} \nabla  e^{-it\L_{a}}W^{J}_{n}  \|_{L^{5}_{t}{L}_{x}^{\frac{15}{8}}}\leq
\|\big(\sum^{J}_{j=1}\psi^{j}_{n}\big) \nabla  e^{-it\L_{a}}W^{J}_{n}  \|_{L^{5}_{t}{L}_{x}^{\frac{15}{8}}}\\
+\|  e^{-it\L_{a}}W^{J}_{n} \|_{L^{10}_{t, x}}
\| \nabla e^{-it\L_{a}}W^{J}_{n} \|_{L^{10}_{t}{L}_{x}^{\frac{30}{13}}}.
\end{split}\]
Thus, using   Strichartz inequality, \eqref{Reminder} and \eqref{Pv22} we see that
\[
\lim_{J\to J^{\ast}}\limsup_{n\rightarrow\infty}
\|u^{J}_{n} \nabla  e^{-it\L_{a}}W^{J}_{n}  \|_{L^{5}_{t}{L}_{x}^{\frac{15}{8}}}
\leq 
\lim_{J\to J^{\ast}}\limsup_{n\rightarrow\infty}
\|\big(\sum^{J}_{j=1}\psi^{j}_{n}\big) \nabla  e^{-it\L_{a}}W^{J}_{n}  \|_{L^{5}_{t}{L}_{x}^{\frac{15}{8}}}.
\]
On the other hand, it follows from \eqref{ImporBound} that 
\[\begin{split}
\|\sum^{J}_{j=J'}\psi^{j}_{n}  \|^{2}_{L^{10}_{t,x}}&\lesssim \sum^{J}_{j=J'}\| \psi^{j}_{n}  \|^{2}_{L^{10}_{x}}+
\sum_{j\neq k}\| \psi^{j}_{n} \psi^{k}_{n}  \|_{L^{5}_{t,x}}
\lesssim_{\tau_{c}}\sum^{J}_{j=J'}E_{a}( \psi^{j}_{n} )+\sum^{J}_{j\neq k}o(1)
\end{split}
\]
as $n\to \infty$. Thus, applying \eqref{Pv22}, H\"older and Strichartz we infer that there 
exists $J'=J(\eta)$ such that
\[\begin{split}
\limsup_{n\rightarrow\infty}
\|\big(\sum^{J}_{j=J'}\psi^{j}_{n}\big) \nabla  e^{-it\L_{a}}W^{J}_{n}  \|_{L^{5}_{t}{L}_{x}^{\frac{15}{8}}}
&\lesssim 
\limsup_{n\rightarrow\infty}[\|\sum^{J}_{j=J'}\psi^{j}_{n}  \|^{2}_{L^{10}_{t,x}}
\| \nabla e^{-it\L_{a}}W^{J}_{n} \|_{L^{10}_{t}{L}_{x}^{\frac{30}{13}}}]\\
&\lesssim \eta \quad \text{uniformly in $J\geq J'$}
\end{split}\]
for any $\eta>0$. In particular, to establish \eqref{FinalLimit}  it suffices to show
\begin{equation}\label{Newfinallimit}
\limsup_{n\rightarrow\infty}
\|\psi^{j}_{n}\nabla  e^{-it\L_{a}}W^{J}_{n}  \|_{L^{5}_{t}{L}_{x}^{\frac{15}{8}}}=0
\quad \text{for all $1\leq j \leq J'$.}
\end{equation}
To this end, we observe that for any $\delta>0$ there exists $\varphi^{j}_{\delta}\in C^{\infty}_{\delta}$
with support in $[-T,T]\times\left\{|x|\leq R\right\}$ such that (see \eqref{BewAprox11})
\[
\|  \psi^{j}_{n}-(\lambda^{j}_{n})^{-\frac12} \varphi^{j}_{\delta}\big( \tfrac{t}{(\lambda_n^j)^2}+t^{j}_{n},
\tfrac{x-x^{j}_{n}}{\lambda^{j}_{n}}\big)   \|_{L^{10}_{t}{\dot{H}}_{x}^{1,\frac{30}{13}}}\leq \delta.
\]
Writing 
\[
\tilde{W}^{J}_{n}(t,x):=(\lambda^{j}_{n})^{\frac{1}{2}}[e^{-it\L_{a}}W^{J}_{n}]( (\lambda^{j}_{n})^{2}(t-t^{j}_{n}),
 \lambda^{j}_{n}x+ x^{j}_{n})
 \]
 and applying Lemma~\ref{LocalSmoo} (local smoothing), equivalence of Sobolev norms, and H\"older's inequality, we finally obtain
\begin{align*}
\|\psi^{j}_{n}&\nabla  e^{-it\L_{a}}W^{J}_{n}  \|_{L^{5}_{t}{L}_{x}^{\frac{15}{8}}} \\
	&\lesssim \delta
	\| \nabla e^{-it\L_{a}}W^{J}_{n} \|_{L^{10}_{t}{L}_{x}^{\frac{30}{13}}}
	+	\|\varphi^{j}_{\delta}\|_{L^{\infty}_{x}}
	\| \nabla \tilde{W}^{J}_{n} \|_{L^{5}_{t}{L}_{x}^{\frac{15}{8}}([-T,T]\times\left\{|x|\leq R\right\})}\\
	&\lesssim\delta+C(\delta,T,R)[ \|  e^{-it\L_{a}}W^{J}_{n} \|^{\frac{1}{32}}_{L^{10}_{t, x}}
	\| W^{J}_{n}  \|^{\frac{31}{32}}_{\dot{H}^{1}_{x}}
	+	\|  e^{-it\L_{a}}W^{J}_{n} \|^{\frac{1}{28}}_{L^{10}_{t, x}}
	\| W^{J}_{n}  \|^{\frac{27}{28}}_{\dot{H}^{1}_{x}}].
	\end{align*}
Thus \eqref{Newfinallimit} finally follows from \eqref{Reminder}, which completes the proof of  \eqref{Bound33}.
\end{proof}

As described above, with \eqref{Bound22} and \eqref{Bound33}, we complete the preclusion of Scenario~2 and hence  the proof of Proposition \ref{PScondition}.
\end{proof}

\section{Preclusion of compact solutions}\label{Su8}

In this section, we use the localized virial argument to preclude the possibility of a solution $u_c$ as in Theorem~\ref{CompacSolution}, thus completing the proof of Theorem~\ref{MainTheorem}. 

We begin with the following result.
\begin{proposition}\label{BuIo} Suppose $u_{c}$ is a solution as in Theorem \ref{CompacSolution}. Then for every $\epsilon>0$ there exists $R=R(\epsilon)>1$ such that
\begin{equation}\label{Uniform}
\sup_{t\in \R}\int_{|x|>R}|\nabla u_{c}(t,x)|^{2}+| u_{c}(t,x)|^{2}+| u_{c}(t,x)|^{4}+| u_{c}(t,x)|^{6}dx\leq \epsilon.
\end{equation}
Moreover, there exists $\eta>0$ such that
\begin{equation}\label{BoundVirial}
\begin{split}
V_{a}(u_{c}(t))\geq\eta, \quad \mbox{for all $t\in \R$},
\end{split}
\end{equation}
where $V_a$ is the virial functional defined in \eqref{E:VFnl}.
\end{proposition}
\begin{proof} The bound \eqref{Uniform} follows immediately from compactness, Gagliardo--Nirenberg, and Sobolev embedding. Next, suppose \eqref{BoundVirial} fails.  Then there exist $t_n\in\R$ such that $\lim_{n\rightarrow\infty}V_{a}(u_{c}(t_{n}))=0$.  By compactness, there then exists $u^*\in H^1$ so that that $u(t_{n})\rightarrow u^{\ast}$ strongly in $H_{x}^{1}(\R^{3})$ along some subsequence. By continuity of $\F$ and $V_{a}$, we deduce that
\[
V_{a}(u^{\ast})=0 \quad \mbox{and} \quad \F(u^{\ast})=\F(u_{c}(t_{n}))=\tau_{c}<\infty, 
\]
contradicting Lemma \ref{FunctionF}(ii). \end{proof}

\begin{proof}[{Proof of Theorem~\ref{MainTheorem}}]  We suppose Theorem~\ref{MainTheorem} fails and take a solution $u_c$ as in Theorem~\ref{CompacSolution}.  We now use the virial identity: writing
\[
I(t)=\int_{\R^{3}}\phi(x)|u_{c}(t,x)|^{2}\,dx
\]
for a radial function $\phi$ to be specified below, we use \eqref{NLS} to compute 
\begin{align*}
	\partial_{t}I(t)&=2\IM \int_{\R^{3}}\nabla\phi\cdot\nabla u_{c} \bar{u_{c}}\,dx, \\
  \partial_{tt}(t)& =\int_{\R^{3}}\bigl[4 \RE\nabla \bar{u_{c}}\cdot\nabla^{2}\phi\,\nabla u_{c} \\
&\quad	+4|u_{c}|^{2}\tfrac{ax}{|x|^{4}}\cdot \nabla\phi-
	\Delta\phi|u_{c}|^{4}+\tfrac{4}{3}\Delta\phi|u_{c}|^{6} -\Delta^{2}\phi\,|u_{c}|^{2}\bigr]\,dx.
\end{align*}
As $\phi$ is radial, we may rewrite this as 
\begin{align*}
  \partial_{tt}I(t)&=4\int_{\R^{3}}\tfrac{\phi'}{r}|\nabla u_{c}|^{2}dx+4\int_{\R^{3}}\bigl(\tfrac{\phi''}{r^{2}}-\tfrac{\phi'}{r^{3}}\bigr)|x\cdot\nabla u_{c}|^{2}dx\\
	&\quad +\int_{\R^{3}}\bigl(\phi''(r)+\tfrac{2}{r}\phi'(r)\bigr)(\tfrac{4}{3}|u_{c}|^{6}-|u_{c}|^{4})dx\\
	&\quad -\int_{\R^{3}}\Delta^{2}\,\phi|u_{c}|^{2}dx+4a\int_{\R^{3}}\tfrac{\phi'}{r}\tfrac{|u_{c}|^{2}}{|x|^{2}}dx.
\end{align*}
We now specialize to the choice $\phi(x)=R^{2}\psi(\frac{|x|}{R})$, where $\psi$ satisfies
\[
\psi(r)=
\begin{cases} 
r^{2}, \quad 0\leq r\leq R;\\
0, \quad r\geq 2R,
\end{cases}  
\quad
0\leq\psi\leq r^{2}, \quad \psi''\leq 2,\quad \psi^{(4)}\leq\tfrac{4}{R^{2}}.
\]
For this choice of $\phi$, the identity above yields  
\begin{equation}\label{Vzero11}
\begin{split}
\partial_{tt}I(t)\geq 8V_{a}(u_{c}(t))-\mathcal{O}\(\int_{|x|\geq R}\bigl[|\nabla u_{c}|^{2}+|u_{c}|^{2} +|u_{c}|^{4}+|u_{c}|^{6}\bigr](t,x)dx \).
\end{split}
\end{equation}

Now, applying \eqref{Uniform} with $R$ sufficiently large and using \eqref{BoundVirial}, we deduce
\begin{equation}\label{ContraV}
\partial_{tt}I(t)\geq \eta>0 \quad \text{uniformly for $t\in [0, \infty)$}.
\end{equation}
On the other hand,
\[
\left|\partial_{t}I(t)\right|\lesssim R\|u_{c}\|^{2}_{L^{\infty}_{t}H^{1}_{x}}\lesssim_{\tau_{c}} R.
\]
Thus, the Fundamental Theorem of Calculus implies 
\[
\eta T\lesssim\left|\int^{T}_{0}\partial_{tt}I(t)dt\right|\lesssim_{\tau_{c}} R\qtq{for any} T>0,
\]
which yields a contradiction for $T$ sufficiently large.\end{proof}

\end{document}